\documentclass[11pt]{amsart}

\usepackage{amsmath}
\usepackage{amssymb}
\usepackage{amsthm}
\usepackage{geometry}
\usepackage{graphicx}
\usepackage{color}
\usepackage{hyperref}
\usepackage[utf8]{inputenc}
\usepackage{enumerate}
\usepackage{dsfont}

\newcommand{\R}{\mathbb{R}}
\newcommand{\Q}{\mathbb{Q}}
\newcommand{\N}{\mathbb{N}}
\newcommand{\Z}{\mathbb{Z}}
\newcommand{\eps}{\varepsilon}
\newcommand{\Eps}{\mathcal{E}}
\newcommand{\V}{\mathcal{V}}
\newcommand{\X}{\mathcal{X}}
\newcommand{\Y}{\mathcal{Y}}
\newcommand{\Zz}{\mathcal{Y}}
\renewcommand{\d}{\,\mathrm d}
\DeclareMathOperator{\lcm}{lcm}
\DeclareMathOperator{\meas}{meas}
\DeclareMathOperator{\sgn}{sgn}
\DeclareMathOperator{\odd}{odd}

\newtheorem{thm}{Theorem}[section]
\newtheorem{cor}[thm]{Corollary}
\newtheorem{prop}[thm]{Proposition}
\newtheorem{lem}[thm]{Lemma}
\newtheorem{defin}[thm]{Definition}
\theoremstyle{definition}
\newtheorem{remark}[thm]{Remark}
\newtheorem*{acknowledgements}{Acknowledgements}

\numberwithin{equation}{section}

\begin{document}
\title{Greedy approximations by signed harmonic sums and the Thue--Morse sequence}

\author[S.~Bettin]{Sandro Bettin}
\address[S.~Bettin]{Dipartimento di Matematica\\
  Universit\`{a} di Genova\\
  Via Dodecaneso 35\\
  16146 Genova\\
  Italy}
\email{bettin@dima.unige.it}

\author[G.~Molteni]{Giuseppe Molteni}
\address[G.~Molteni]{Dipartimento di Matematica\\
  Universit\`{a} di Milano\\
  Via Saldini 50\\
  20133 Milano\\
  Italy}
\email{giuseppe.molteni1@unimi.it}

\author[C.~Sanna]{Carlo Sanna}
\address[C.~Sanna]{Dipartimento di Matematica\\
  Universit\`{a} di Genova\\
  Via Dodecaneso 35\\
  16146 Genova\\
  Italy}
\email{carlo.sanna.dev@gmail.com}

\date{\today}

\begin{abstract}
Given a real number $\tau$, we study the approximation of $\tau$ by signed harmonic sums $\sigma_N(\tau)
:= \sum_{n \leq N}{s_n(\tau)}/n$, where the sequence of signs $(s_N(\tau))_{N \in\N}$ is defined
``greedily'' by setting $s_{N+1}(\tau) := +1$ if $\sigma_N(\tau) \leq \tau$, and $s_{N+1}(\tau) := -1$
otherwise. More precisely, we compute the limit points and the decay rate of the sequence
$(\sigma_N(\tau)-\tau)_{N \in \N}$. Moreover, we give an accurate description of the behavior of the
sequence of signs $(s_N(\tau))_{N\in\N}$, highlighting a surprising connection with the Thue--Morse
sequence.
\end{abstract}

\maketitle

\section{Introduction}
Riemann's rearrangement theorem~\cite[\S1.1]{MR1108619} asserts that any conditionally convergent series can be rearranged to
converge to any given  $\tau\in\R\cup\{\pm\infty\}$. The classical proof of this result is constructive:
assuming $\tau\in\R$, one first sums the positive values until they exceed $\tau$, then one sums the first
few negatives values until going below $\tau$, and so forth. Much in the same way, one can show that
for all $\tau\in\R$ there exist sequences $(s_n)_{n\in\N}$ with $s_n\in\{\pm 1\}$ such that
\begin{equation*}
\sum_{n=1}^\infty\frac{s_n}{n} = \tau.
\end{equation*}
A natural question is then to determine whether one can find sequences $(s_n)_{n\in\N}$ such that the
partial sums $\sigma_{n}:=\sum_{m=1}^n\frac{s_m}{m}$ converge to $\tau$ particularly quickly. The rate of
convergence of $\sigma_n$ to $\tau$, however, has a clear limitation, since
$|\sigma_n-\sigma_{n-1}|=\frac{1}{n}$ for all $n\in\N$, and so $|\sigma_{n-1}-\tau|\geq \frac{1}{2n}$ for
infinitely many integers $n$. One can then modify the problem along two different routes. The first
approach is to allow the sequence to change and study the  asymptotic behavior of
\begin{equation*}
\mathfrak{m}_N(\tau):= \min\!\left\{\left|\sigma_N-\tau\right| \colon s_1,\dots,s_N\in\{\pm 1\}\right\}.
\end{equation*}
This approach was considered by the authors~\cite{BettinMolteniSanna1}, where it is shown that
\begin{equation}\label{eq:mN_ineq}
\mathfrak{m}_{N}(\tau) < K_{\tau,\eps} \exp\!\left(-\frac1{\log 4-\eps} (\log N)^2\right)
\end{equation}
for every $\eps>0$ and some constant $K_{\tau,\eps}>0$.
This inequality was obtained by interpreting the problem probabilistically, studying the rate of convergence
in distribution of the random variable $\sum_{n=1}^\infty\frac{Z_n}{n}$, where the $Z_n$ are independent
uniformly distributed random variables in $\{-1,+1\}$.

The sequence of signs realizing the minimums $\mathfrak{m}_N(\tau)$ and $\mathfrak{m}_{N+1}(\tau)$
are not related, in general. In particular, the first one is not a subsequence of the second one, and
a universal sequence $(s_n)_{n\in\N}$ giving the best approximation does not exist. Rather than taking the minimum over
all possible choices for the $(s_n)_{n\in\N}$, the second approach is then to fix a specific sequence
$(s_n)_{n\in\N}$ and to determine whether $|\sigma_{n}-\tau|$ can go to zero extremely quickly with $n$ running along
a subsequence $(n_k)_{k\in\N}$. A natural choice for a sequence $\big(s_n(\tau)\big)_{n\in\N}$ with
\begin{equation*}
\sigma_{n}(\tau):=\sum_{m=1}^n\frac{s_m(\tau)}{m}
\end{equation*}
(and $\sigma_0(\tau):=0$) converging to $\tau$ is the following:
\begin{equation}\label{eq:def_sn}
s_n(\tau):=\begin{cases}
1 , & \text{if }\tau  \geq \sigma_{n-1}(\tau) ;\\
-1 , & \text{if }\tau< \sigma_{n-1}(\tau) .
\end{cases}
\end{equation}
In other words, analogously to the proof of Riemann's rearrangement theorem, one chooses $s_n(\tau)=-1$
if the partial sum $\sigma_{n-1}(\tau)$ exceeds $\tau$, and $s_n(\tau)=1$ otherwise (see
Section~\ref{sec:2} below for other possible choices when $\sigma_{n-1}(\tau)=\tau$). We call this a
\emph{greedy approximation to $\tau$ by signed harmonic sums} since at every step, given
$\sigma_{n-1}(\tau)$, the value of $s_n(\tau)$ is chosen so that the distance of $\sigma_{n}(\tau)$ from
$\tau$ is minimized.

In this paper we consider the second approach, analyzing the decay of $\sigma_n(\tau)-\tau$.
Thus, we  stress that from now on, with $s_n(\tau)$ we mean the deterministic sequence defined
by~\eqref{eq:def_sn}. Surprisingly, the behavior of the sequence $s_n(\tau)$ is not chaotic, but is
extremely structured and allows one to prove precise results on the asymptotic behavior of
$\sigma_n(\tau)-\tau$.

Our first result determines the set $S_k'(\tau)$ of limit points of the sequence
\begin{equation*}
S_k(\tau):=\big((\sigma_n(\tau)-\tau)\cdot n^k\big)_{n\in\N}
\end{equation*}
for all integers $k\in\N$ (the case $k=0$ being trivial, since the sequence tends to $0$). Surprisingly,
only a few limit points can arise.

\begin{thm}\label{thm:main_theorem}
There exist sets $X_1,X_2,\dots$ which are non-empty, pairwise disjoint, countable and with
$X_h\cap\overline \Q=\emptyset$ for all $h$, such that for all integers $k\in\Z_{\geq 0}$ and all
$\tau\in\R$ one has
\begin{equation}\label{eq:main_theorem}
S_{k+1}'(\tau)
=
\begin{cases}
 \{0,\pm c_0\}              , & \textup{if $k=0$ and $\tau \notin X_{1}$};       \\
 \{\pm c_0/ 2\}             , & \textup{if $k=0$ and $\tau \in X_{1}$};          \\
 \{0,\pm c_k,\pm \infty\}   , & \textup{if $k\geq 1$ and $\tau \notin Y_{k+1}$}; \\
 \{\pm c_k / 2,\pm \infty\} , & \textup{if $k\geq 1$ and $\tau \in X_{k+1}=Y_{k+1}\setminus Y_k$};\\
 \{\pm \infty\}             , & \textup{if $k\geq 1$ and $\tau \in Y_k$};
\end{cases}
\end{equation}
where $Y_k:=\cup_{h\leq k}X_h$ and $c_k:= 2^{\binom{k}{2}}k!$.
\end{thm}
Let $Y_{\infty}:=\cup_{h=1}^\infty X_h$. Notice that $Y_{\infty}$ is countable, and thus in particular has
Lebesgue measure zero. From Theorem~\ref{thm:main_theorem} one immediately deduces the following
corollary.
\begin{cor}\label{cor:main_thm_special_case}
For all integers $k\geq 0$ and all $\tau\in\R$ one has $S_{k+1}'(\tau)\subseteq\{0,\pm c_k,\pm c_k/2,\pm
\infty\}$. Moreover, if $\tau\notin Y_{\infty}$ one has $S_{k+1}'(\tau)=\{0,\pm c_k,\pm \infty\}$ for all
$k\geq 1$.
\end{cor}
One can also obtain sharp explicit bounds for $\sigma_{n}(\tau) - \tau$.
\begin{cor}\label{cor:inequalities}
Let $k\geq 0$ and let $\tau\notin Y_k$. Then there are infinitely many $n$ such
that
\begin{equation}\label{eq:ineq1}
0 < (\sigma_{n}(\tau) - \tau)n^{k+1} < c_k,\\
\end{equation}
and infinitely many $n$ such that
\begin{equation}\label{eq:ineq2}
0 < (\tau - \sigma_{n}(\tau))n^{k+1} < c_k .
\end{equation}
In particular, if $\tau\notin Y_{\infty}$, then there is a subsequence of
$\big(\sigma_n(\tau)-\tau\big)_{n\in\N}$ that decays faster than any power of $n$.
\end{cor}
\begin{remark}
Corollary~\ref{cor:inequalities} is slightly sharper than what follows from
Theorem~\ref{thm:main_theorem}. Its proof also shows that the $n$ satisfying one of~\eqref{eq:ineq1}
and~\eqref{eq:ineq2} form an arithmetic progression modulo $2^k$ when $n$ is sufficiently large.
\end{remark}

\begin{figure}[h]
\includegraphics[width=0.485\textwidth]{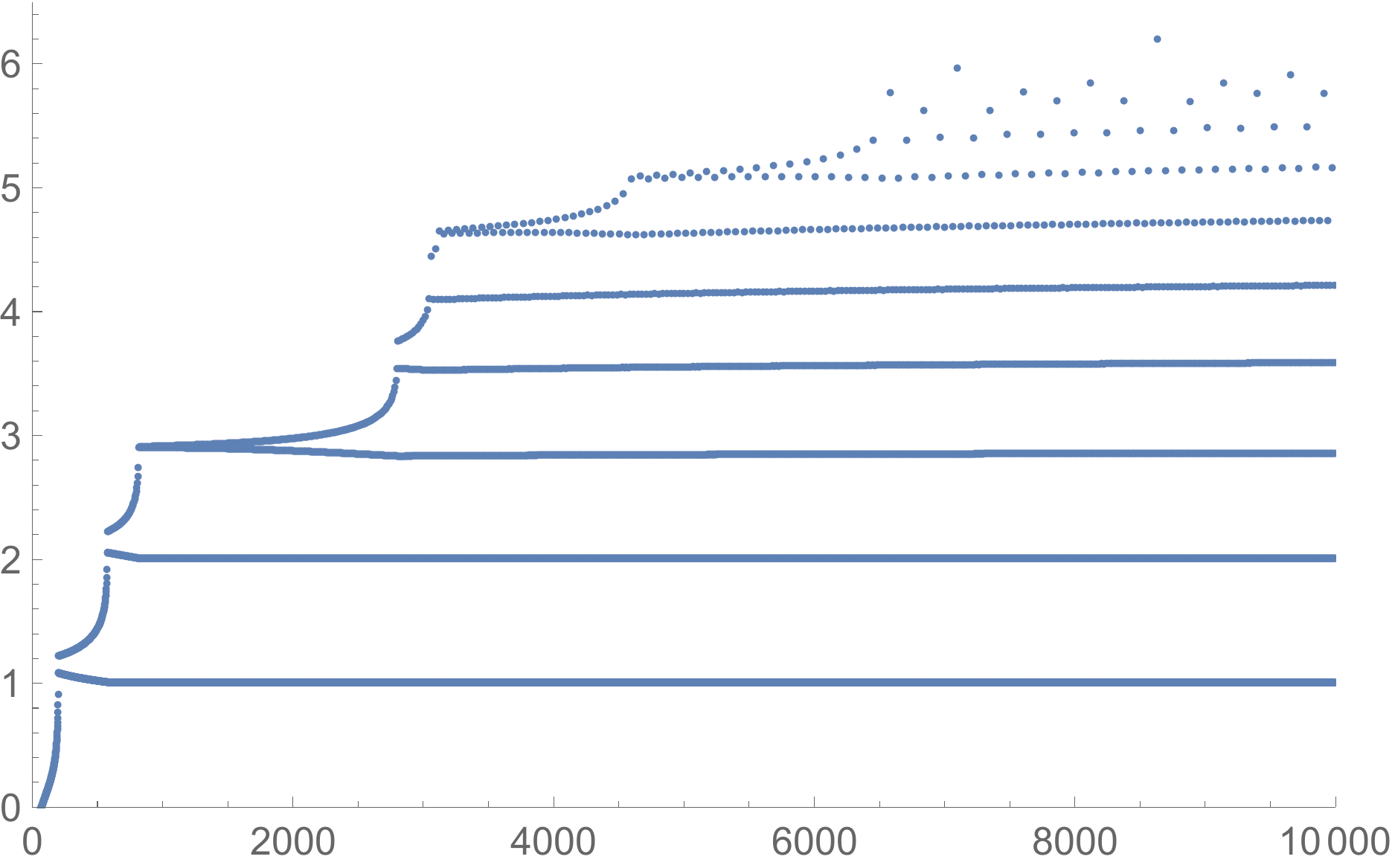}
\hspace{0.2cm}
\includegraphics[width=0.485\textwidth]{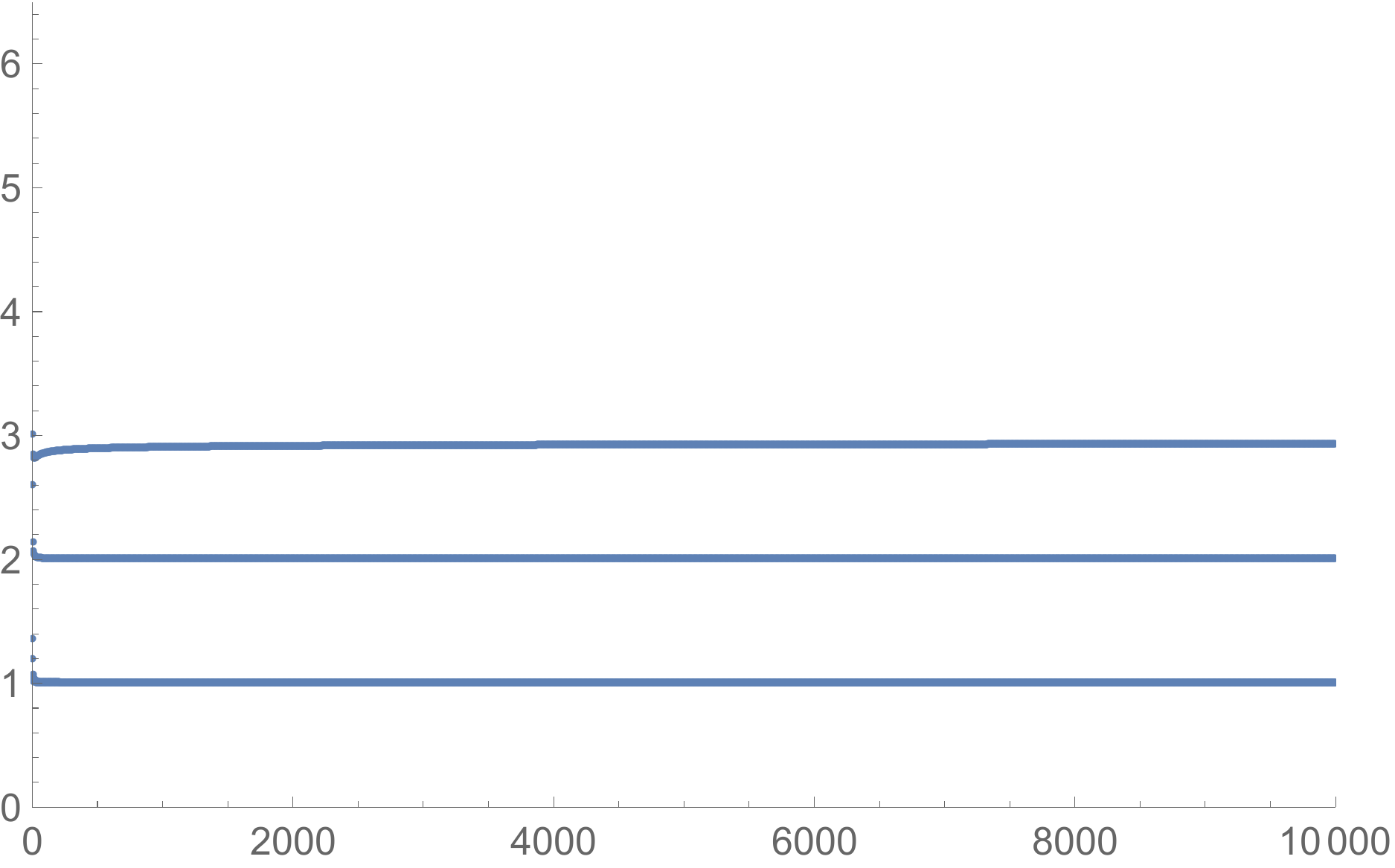}
\caption{The graphs of $-\frac{\log|\sigma_n(\tau)-\tau|}{\log n}$ for $2\leq n\leq 10^4$ for
         $\tau=\sqrt{2}+2\sqrt{5}\notin Y_{\infty}$ (left)
         and $\tau = U_{3,2} = \frac{1}{8}((2\sqrt{2}-1) \pi + \log 4)\in X_3$ (right).}
\label{fig:1}
\end{figure}

For almost all $\tau$ one can even determine exactly how small $|\tau - \sigma_{n}(\tau)|$ can be.
Indeed, one obtains that, for almost all $\tau$, the distance $|\tau - \sigma_{n}(\tau)|$ is infinitely
often as small as $e^{-(\log n)^2/\log 4(1+o(1))}$ and that this bound is optimal. It is quite remarkable that
the sequence $s_n(\tau)$ allows one to recover, albeit for a subsequence and only for almost all $\tau$,
exactly the estimate~\eqref{eq:mN_ineq} obtainable with probabilistic
methods~\cite{BettinMolteniSanna1}.
\begin{thm}\label{thm:asymp}
For almost all $\tau\in\R$ one has
\begin{equation}\label{eq:asymptau}
\liminf_{n\to\infty}\frac{\log|\tau - \sigma_{n}(\tau)|}{(\log n)^2}
= -\frac{1}{\log 4}.
\end{equation}
\end{thm}
\begin{remark}
Equation~\eqref{eq:asymptau} does not hold for all $\tau\in\R$. For example, by
Theorem~\ref{thm:main_theorem} one has that~\eqref{eq:asymptau} does not hold for all $\tau\in
Y_{\infty}$. Indeed, if $\tau\in X_k$ one has that $|\tau - \sigma_{n}(\tau)|$ is always larger than a
constant times $n^{-k}$. In the opposite direction, in Proposition~\ref{prop:limitation} below we shall
show that for any function $f\colon\N\to\R_{>0}$ there exists a real number $\tau$ such that $|\tau -
\sigma_{n}(\tau)|<f(n)$ infinitely often.
\end{remark}

\begin{figure}[h]
\includegraphics[width=0.585\textwidth]{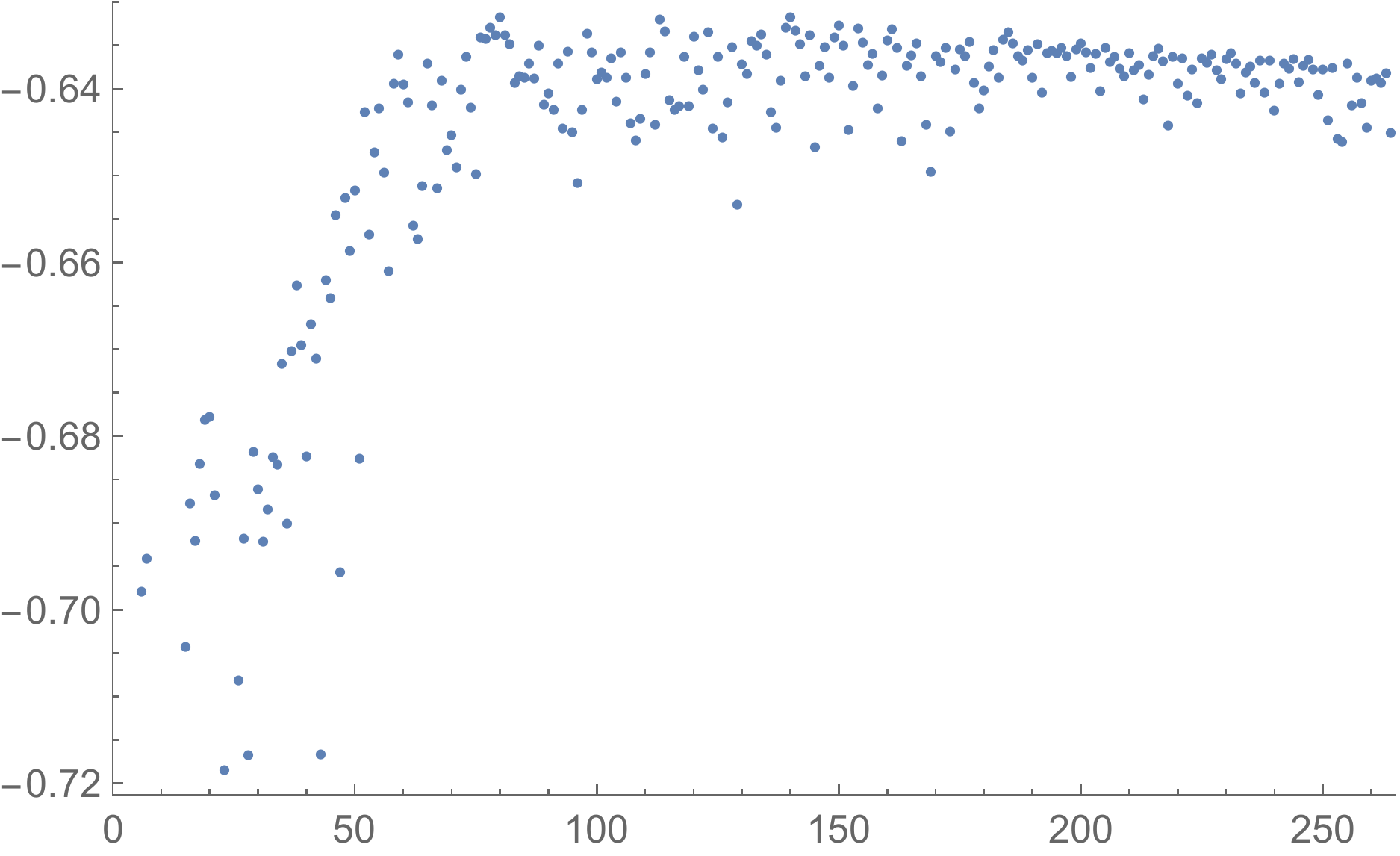}
\caption{The plot of $\frac{\log|\sigma_{m_j}(\tau)-\tau|}{(\log m_j)^2}$ for $\tau=0$ and $1\leq
         j\leq 265$ (i.e., $m_j<10^{10}$), where $m_j$ is the minimal integer such that
         $|\sigma_{m_j}(\tau)-\tau| < |\sigma_{m}(\tau)-\tau|$ for all $m<m_j$. Note that the sequence
         is not far from $-1/{\log 4}\approx -0.721\dots$ which Theorem~\ref{thm:asymp} predicts as the
         liminf for almost every $\tau$.}
\label{fig:2}
\end{figure}

The proofs of Theorem~\ref{thm:main_theorem} and Theorem~\ref{thm:asymp} are based on a surprising
connection between the sequence $(s_n({\tau}))_{n\in\N}$ and the Thue--Morse sequence $(t_n)_{n\geq 0}$.
This is a binary sequence whose first few values are
\begin{equation*}
0, 1, 1, 0, 1, 0, 0, 1, 1, 0, 0, 1, 0, 1, 1, 0, 1, 0, 0, 1, \ldots
\end{equation*}
and which can be defined in several equivalent ways. For example, it can be defined by setting $t_n := 0$
(respectively, $t_n := 1$) if the binary expansion of $n$ has an even (respectively, odd) number of $1$s.
Alternatively, it can be defined by the recurrence relation
\begin{equation*}
t_0 :=0,
\quad
t_{2n} = t_n,
\quad
t_{2n + 1} = 1 - t_n,
\end{equation*}
or by the $L$-system which starts with $0$ and at each step substitutes the digits $0$ and $1$ with $0,1$
and $1,0$, respectively (see~\cite{DekkingMendesFrancePoortenI,MR0685557,MR0685559,Erratum_MR0685557} and~\cite[Section~2.2]{Lothaire1}).
Furthermore, the Thue--Morse sequence is $2$-automatic, meaning that $t_n$ is obtained by feeding a deterministic finite automaton with an output function with the base-$2$ representation of $n$~\cite[Example~5.1.2]{MR1997038}.
The Thue--Morse sequence has repeatedly appeared in several fields of mathematics, including dynamic
systems, combinatorics, number theory and approximation theory, and has been studied extensively in its
many aspects. We refer for example to~\cite{Allouche, Dekking1, DubucElqortobi, Euwe,  KomornikLoreti,
MauduitRivat2, MauduitRivat, Morse, Prouhet} and to the nice survey of
Allouche and Shallit~\cite{AlloucheShallit} for a more extensive discussion of the ubiquity of the
Thue--Morse sequence.

For our purposes, it is more convenient to modify the Thue--Morse sequence so that it takes its values in
$\{\pm 1\}$, and thus we set $\eps_n := (-1)^{t_n}$ for all $n\geq 0$.

The first connection between the Thue--Morse and our greedy sequence is apparent from the following
identity~\cite{robbins,woods}  (see~\cite{MR3942297,AlloucheShallitjl,Shallitjnt} for some generalizations)
\begin{equation*}
\prod_{n=0}^\infty \bigg(1-\frac{1}{2n+2}\bigg)^{\eps_{n}}
= \frac{1}{\sqrt{2}}.
\end{equation*}
Shallit observed, and Allouche and Cohen~\cite{AlloucheCohen} later proved, that for all $n$ one has
$\prod_{n=0}^{j-1}(1-\frac{1}{2n+2})^{\eps_n} > 1/\sqrt{2}$ if and only if $\eps_j=1$. In other words,
passing to the logarithms, $\sum_{n=0}^\infty\eps_nf(n)$ is the greedy series for $-\frac{1}{2}\log2$
with respect to the weight function $f(x):=\log(1-\frac{1}{2x+2})$.

It turns out that the connection between the sequence of signs $s_{n}(\tau)$ and the Thue--Morse sequence is much broader than the result of Allouche and Cohen might suggest.
Indeed, $s_n(\tau)$ can be written in terms of the Thue--Morse sequence for all $\tau$. We describe this connection in
the following result, which is the key ingredient in the proof of Theorem~\ref{thm:main_theorem}.\\
We let $B_{r}$ denote the block $(\eps_{n})_{0\leq n< 2^r}$ for all $r\in\Z_{\geq 0}$.
Also, given two or more vectors $v_{i}=(v_{i,j})_{j}$, with $(v_1,v_2,\dots)$ we mean the vector (or
infinite sequence) obtained by concatenating the vectors $v_1,v_2,\dots$
Moreover, given a vector $v$ and a scalar $\kappa$, with $\kappa\cdot v$ we mean the vector $v$ multiplied by the scalar $\kappa$; for
example, $(-1)\cdot B_r=(-\eps_{n})_{0\leq n< 2^r}$.
\begin{thm}\label{thm:blocks}
Let $\tau\in\R$. Then there exist a non-decreasing sequence $(k_i)_{i\in\N}$, with $k_i\in\Z_{\geq 0}$ for
all $i$, and a sequence $(\kappa_i)_{i\in\N}$, with $\kappa_i\in\{\pm 1\}$, such that
\begin{equation}\label{eq:blocks_expansions}
(s_n(\tau))_{n\in\N}=(\kappa_1\cdot B_{k_1}, \kappa_2\cdot  B_{k_2},\kappa_3\cdot B_{k_3},\dots).
\end{equation}
Moreover,  $\lim_{i\to\infty}k_i=+\infty$ if $\tau\notin Y_{\infty}$, whereas if $\tau\in X_k$ for some
$k$, then $k_i=k$ and $\kappa_i=1$ for $i$ large enough.
\end{thm}
Since for all $r_1,r_2\in\Z_{\geq 0}$ with $r_1\leq r_2$ the block $B_{r_2}$ contains the block $B_{r_1}$,
we immediately deduce the following corollary.
\begin{cor}\label{cor:corblo}
If $\tau\notin Y_{\infty}$, then for all $r\geq 0$ the sequence $(s_n(\tau))_{n\in\N}$ contains the block
$B_{r}$ infinitely many times.
\end{cor}
We remark that Theorem~\ref{thm:blocks} characterizes the elements of the exceptional sets $X_k$ as the
real numbers such that the expansion~\eqref{eq:blocks_expansions} is eventually periodic with repeating
block $B_k$ (in fact, this will actually be our definition of $X_k$).  In particular, the elements in
$X_k$ can be written as a rational number plus $U_{k,m}$ for some $m \in [0,2^k)$, where
\begin{equation}\label{eq:def_U}
U_{k,m}:=\sum_{n=1}^{+\infty}\frac{f_{k}(n+m)}{n},
\end{equation}
with $f_{k}$ being the $2^k$-periodic function such that $f_{k}(n)=\eps_{n}$ for all $n \in [0, 2^k)$.
(Notice that $U_{k,m}=-U_{k,m+2^{k-1}}$ for all $m \in [0,2^{k-1})$.) For small $k$ one can easily write
the values $U_{k,m}$ explicitly. For example,
\begin{equation}\label{eq:u01}
U_{1,0}=-\log 2, \qquad
U_{2,0}=-\tfrac14 (\pi + \log 4),\qquad U_{2,1}=-\tfrac14 (\pi - \log 4).
\end{equation}
and so, in particular,
\begin{equation*}
X_1\subseteq \Q\pm \log 2,
\qquad  X_2\subseteq \Q\pm  \tfrac14 (\pm \pi \pm \log 4).
\end{equation*}
In general the constants $U_{k,m}$ can be written in terms of logarithms of cyclotomic units in
$\Q(\xi)$, where $\xi$ is a primitive $2^k$-root of unity. Baker's theorem can then be used to show that
$U_{k,m}$ is transcendental for all choices of $k$ and $m$.  We refer to Section~\ref{sec:exceptional}
for more details on the sets $X_k$ and the constants $U_{k,m}$.

Since the Thue--Morse sequence plays a special role in our work, we examine the
case of the constant
\begin{equation*}
\tau_0 := \sum_{n=1}^{+\infty} \frac{\eps_{n-1}}{n},
\end{equation*}
in more detail. This constant appears as the value at $s=1$ of the Thue-Morse
Dirichlet series $\sum_{n=1}^{\infty}\frac{\eps_{n-1}}{n^s}$, which was studied
in~\cite{AlloucheCohen,ACM-FS}. In Section~\ref{sec:A8} we prove that this series converges, and that the
sequence $(\eps_{n})_{n\geq 1}$ indeed coincides with the sequence of signs produced by the greedy
algorithm (much as in the aforementioned work of Allouche and Cohen~\cite{AlloucheCohen}). In other
words, $\eps_{n-1}=s_n(\tau_0)$ for all $n\in\N$. We remark that this is not immediate. Since the
Thue--Morse sequence is not eventually periodic, then $\tau_0\notin Y_\infty$ and the results in
Corollaries~\ref{cor:main_thm_special_case}--\ref{cor:inequalities} apply to $\tau_0$. However, since we
know the sequence $(s_n(\tau_0))_{n\in\N}$ exactly, we are able to prove these results in an explicit form.
\begin{thm}\label{thm:A5}
For all $n\geq 1$ we have $s_{n}(\tau_0)=\eps_{n-1}$. Moreover, let $k\geq 1$, and let $n=2^k n'$ with
$n'$ odd. Then $|\sigma_n-\tau_0|n^k\leq c_{k-1}$, and $(\tau_0-\sigma_n(\tau_0))n^{k+1}\sim_k
\eps_nc_k$ as $n'$ goes to infinity. Finally,~\eqref{eq:asymptau} holds for $\tau=\tau_0$ in the more
precise form
\begin{equation}\label{eq:asymptau0}
\liminf_{n\to\infty}\frac{\log|\tau_0 - \sigma_{n}(\tau_0)|+\frac{1}{\log 4}(\log n)^2}{\log n\log\log n}
=\frac{1}{\log 2}.
\end{equation}
\end{thm}
The first part of this theorem allows one to compute $\tau_0$ very quickly.
The decimal representation of $\tau_0$ with $50$ correct digits is:
\begin{equation*}
\tau_0 = 0.39876108810841881240743054440027306033680891546719\ldots.
\end{equation*}

In this paper we focused on the case of signed harmonic sums, but our work could also be extended with little
effort to signed sums of the form $\sum_{n}s_n^{\alpha}(\tau)\, n^{-\alpha}$ for any $\alpha \in (0,1]$
(and, to some extent, also to more generic weight functions). In particular,
Theorems~\ref{thm:main_theorem},~\ref{thm:asymp} and~\ref{thm:blocks} still hold (removing the statement
$X_h\cap \overline\Q=\emptyset$), once one replaces $S_k(\tau)$ with
$S^{\alpha}_k(\tau):=\big((\sigma_n(\tau)-\tau)\cdot n^{k+\alpha-1}\big)_{n\in\N}$ and $c_k$ with
$c_{k}^{\alpha}:= 2^{\binom{k}{2}}\alpha(\alpha+1)\cdots (\alpha + k-1)$.

One can also generalize the problem to other directions.
For example, one can consider the analogous problems
in higher dimensions or can require that $s_n$ takes its values in the $k$-th roots of unity rather than in
$\{\pm1\}$. For $k\neq 2$ the greedy algorithm again produces a representation for every complex $\tau$.
The rate of convergence has a similar behavior, but some new phenomena appear, and it is possible that
in this case the role of the Thue--Morse sequence is played by its analogue for the base-$k$
representations of numbers. We still need to study these generalizations in depth.

\begin{acknowledgements}
S. Bettin is member of the INdAM group GNAMPA.
G.~Molteni and C.~Sanna are members of the INdAM group GNSAGA. 
The work of the first and second author is partially supported by  PRIN 2015 ``Number Theory and Arithmetic Geometry''.
C.~Sanna is supported by a postdoctoral fellowship of INdAM.
Part of this work was done during a visit, partially supported by INdAM, of the first author to the
Centre de Recherches Math\'{e}matiques in Montr\'eal. He would like to thank this institution for its
hospitality. He would also like to thank M. Radziwi\l\l{} for several useful discussions as well as for
the help with the C code used for Figure~\ref{fig:2}.
The authors would like to thank the anonymous referees for carefully reading the paper and for giving many suggestions that improved its quality.
\end{acknowledgements}

\subsection*{Notation}
For ease of notation, in the following we shall usually omit the dependence of $\tau$ in $s_n(\tau)$ and
$\sigma_n(\tau)$. We stress, however, that the sequences $(\sigma_n)_n$ and $(s_n)_n$ are the sequences
obtained from the greedy algorithm~\eqref{eq:def_sn} applied to $\tau$.
Finally, given a set $I\subseteq\R$, we indicate with $\chi_I(x)$ the characteristic function of $I$.

\section{Preliminary considerations and the proof of Theorem~\ref{thm:main_theorem}}\label{sec:2}
In the definition~\eqref{eq:def_sn} of the sequence $(s_n)_{n}$ we defined $s_n:=+1$ whenever
$\sigma_{n-1}=\tau$. Of course, other choices are also possible, and one could even decide to stop the
algorithm whenever such equality is achieved. Another natural choice would be that of defining
$(s_{n})_{n}$ as in~\eqref{eq:def_sn} when $\tau\geq 0$ and putting instead
\begin{equation*}
s_n(\tau):=\begin{cases}
1  , & \text{if }\tau> \sigma_{n-1}(\tau);\\
-1 , & \text{if }\tau\leq \sigma_{n-1}(\tau);
\end{cases}
\end{equation*}
when $\tau<0$. Notice that this alternative definition would ensure that $(s_n(\tau))_{n}$ is odd in
$\tau$ for all $\tau\neq0$. The conclusions of this paper are essentially independent of the choice made
in these special cases, so we chose the definition~\eqref{eq:def_sn} since it slightly simplifies some
statements. In any case, independently of the choice made, the equality $\sigma_n=\tau$ (which clearly
requires $\tau\in\Q$) could only be achieved at most one time.
\begin{prop}\label{prop:equality}
Let $h\in\Z$ and $k\in\N$ with $(h,k)=1$. Let $(r_n)_n$ be a sequence taking values in $\{\pm 1\}$. Then
there exists at most one $N\in\N$ such that $\sum_{n=1}^Nr_n/n=h/k$. Moreover, if such an $N\geq2$
exists, then it satisfies $N\leq 3\log k$.
\end{prop}
\begin{proof}
Suppose, to get a contradiction, that there exist $N,M\in\N$, with $N<M$, such that $\sum_{n=1}^Nr_n/n =
\sum_{n=1}^M r_n/n = \tau$.  In particular,
\begin{equation}\label{eq:double_eq}
\sum_{n=N+1}^{M} \frac{r_n}{n} = 0.
\end{equation}
Let $2^\nu$ be the greatest power of $2$ dividing some integer in the range $N+1,\ldots,M$. Notice that
we can assume that $M\geq N+2$ (otherwise the sum contains a unique term, which cannot be $0$), and in
particular $\nu\geq 1$. Now there exists a unique $m$ in $N+1,\ldots,M$ that is divisible by $2^\nu$.
Indeed, if there were two such integers $m,m'$ with $m<m'$, then, letting $m=2^\nu g$ with $g$ odd, we
would have that $N+1\leq m<m+2^\nu \leq m'\leq M$.
But $m + 2^\nu = 2^{\nu + 1} \tfrac{g+1}{2}$, which contradicts the maximality of $\nu$.
Writing $\ell := \lcm(n+1,\ldots,n')$, we have that $ \sum_{n=N+1}^{M} r_n\frac{\ell}{n} $ is odd
and hence non-zero, contradicting~\eqref{eq:double_eq}.

To prove that $N\leq 3\log k$ one can proceed in a similar way, observing that if $\sum_{n=1}^{N}
\frac{r_n}{n} = \frac{h}{k}$, then $k$ is divisible by all prime powers in $(N/2,N]$. 
In fact, let $p^\nu$ be any prime power in $(N/2,N]$.
Then $N<2p^\nu$, so that $p^\nu$ is the unique number in $[1,N]$ that is divisible by
$p^\nu$. Hence, when we multiply the sum $\sum_{n=1}^N \frac{r_m}{m}$ by the least common multiply of the
denominators, we get an integer which is congruent to $r_{p^\nu}$ modulo $p^\nu$; in particular, it is not divisible by $p$.
Thus $k\geq
e^{\psi(N)-\psi(N/2)}$, where $\psi(x)$ is the Chebyshev's function. By~\cite[Theorem
8]{RosserSchoenfeld} one easily deduces that $\psi(N)-\psi(N/2)>N/3$ for all integers $N\geq2$ and the
result follows.
\end{proof}

The set $W:=\{\tau\in\R\colon \sigma_N = \tau \text{ for some }N\in\N\}$ does not seem easy to
characterize, similarly to what happens in analogous problems on Egyptian fractions
(cf.~\cite[\S{D11}]{guy}). It is clear, however, that $W\subseteq\Q$, that $W$ is nonempty (for example, it
contains $1$, $1/2=1-1/2$, $1/6=1-1/2-1/3$) and that it is a decidable set of $\Q$, as there is an
algorithm that, for every $\tau\in\Q$, is able to determine if $\tau\in W$ in a finite number of steps.
We shall not consider the problem of studying $W$ in this paper. However, we remark that by
Proposition~\ref{prop:equality} the equality $\sigma_N=\tau$ can be achieved only once for each $\tau$, and
so $W$ does not depend on which version of the greedy algorithm we use.

\subsection{Sequence of signs and the corresponding inequalities}
A sequence of inequalities corresponds to every sequence of signs $(s_n)_{n}$; 
indeed, one has
\begin{equation}\label{eq:signs_to_ineq}
s_{n+1}= +1 \iff \sigma_n \leq \tau
\qquad\text{and}\qquad
s_{n+1}  = -1 \iff \tau < \sigma_n.
\end{equation}
In particular, one has that $s_{n+1}=+1$ and $s_{n'+1}=-1$ with $n'>n$ if and only if $\sigma_n \leq \tau
< \sigma_{n'}$, whereas the reversed inequality holds if $s_{n+1}=-s_{n'+1}=-1$. In both cases
\begin{equation}\label{eq:signs_to_ineq_2}
s_{n+1}\cdot(\tau-\sigma_n)=|\sigma_n-\tau|\leq |\sigma_{n'} - \sigma_n| =\bigg| \sum_{m=n+1}^{n'} s_m/m\bigg|.
\end{equation}
Thus $|\sigma_n-\tau|\leq (n'-n)/(n+1)$ and stronger conclusions may be drawn if one is able to bound the
sum $\sum_{m=n+1}^{n'} s_m/m$ more effectively, which in turn amounts to controlling the sequence of
signs $(s_m)_m$. The following proposition gives a first example of this phenomenon.
\begin{prop}\label{prop:A1n}
Let $\tau\in\R$. Then after the first change of sign the sequence $(s_n)_n$ has no three consecutive
equal terms. As a consequence, $|\sigma_n-\tau|\leq 2/(n+1)$ for all $n$ following the first sign change.
In particular, the series $\sum_{n=1}^\infty s_n/n$ converges to $\tau$.
\end{prop}
\begin{proof}
The definition of the greedy process and the fact that the harmonic series diverges show that the sequence
$(s_n)_{n}$ contains infinitely many changes of sign. In particular, to prove the first assertion of the
proposition it suffices to show that one cannot have
\begin{equation*}
-s_n
=s_{n+1}
=s_{n+2}
=s_{n+3}
=1
\quad \text{or}\quad
-s_n
=s_{n+1}
=s_{n+2}
=s_{n+3}
=-1
\end{equation*}
for $n$ large enough. Indeed, if the former equality is satisfied, then by~\eqref{eq:signs_to_ineq} we
have
\begin{equation*}
\sigma_{n}< \sigma_{n+1}< \sigma_{n+2}\leq \tau < \sigma_{n-1},
\end{equation*}
whence
\begin{equation*}
1/n = \sigma_{n-1}-\sigma_n > \sigma_{n+2}-\sigma_n = 1/(n+2) + 1/(n+1),
\end{equation*}
and this is impossible for $n\geq 2$. Analogously, one excludes the second case.
The bound for $\sigma_n-\tau$, and thus the convergence of the series, follows immediately
from~\eqref{eq:signs_to_ineq_2}.
\end{proof}
Since $|\sigma_{n+1}-\sigma_{n}|=\frac{1}{n+1}$ for all $n$, one cannot significantly improve the
inequality of Proposition~\ref{prop:A1n} for all large enough $n$. However, the signs patterns exhibited
by Theorem~\ref{thm:blocks} allows us to draw stronger conclusions for arbitrarily large $n$. The
following lemma allows one to control the sum $\sum_{m=n+1}^{n'} s_m/m$ when $s_m$ is given by the
$\pm 1$-valued Thue--Morse sequence $\eps_n := (-1)^{t_n}$.  Before stating it, we define the function
$g_k(x)$ for $k\in\Z_{\geq 0}$ and $x>0$ as
\begin{equation}\label{eq:defin_g_k}
g_{k}(x):=\sum_{0\leq \ell< 2^k}\frac{\eps_\ell}{x+\ell}.
\end{equation}
This function, which could also be defined as a $2^{k-1}$-th iteration of the operator $(\Delta_\alpha
f)(x):= f(x) - f(x+\alpha)$ on the function $g_0(x):=\frac{1}{x}$ (cf.~Lemma~\ref{lem:iteration_g_k} below),
will play an important role in the proofs of our results.
\begin{lem}\label{lem:thm_to_thm}
Let $r=\sum_{0\leq j\leq q}2^{h_j}>0$ with $h_0>h_1>\dots> h_q$. Then, as $x\to\infty$, we have
\begin{equation*}
\sum_{0\leq \ell< r}\frac{\eps_\ell}{x+\ell} \sim (-1)^q \frac{c_{h_q}}{x^{h_q+1}}.
\end{equation*}
In particular, for all $k\in\Z_{\geq 0}$ we have
\begin{equation}\label{eq:thm_to_thm_b}
g_{k}(x) \sim  \frac{c_{k}}{x^{k+1}},
\end{equation}
as $x\to+\infty$. Furthermore, for $x\in\R_{>0}$, the function $g_k(x)$ is positive, decreasing and
satisfies $g_k(x)<c_kx^{-k-1}$.
\end{lem}
We postpone the proof of this lemma to Section~\ref{sec:proof_blocks} (cf.~Corollary~\ref{cor:asymp} and
Lemma~\ref{lem:g_signs}). We now show how this lemma, in combination with Theorem~\ref{thm:blocks}, 
implies Theorem~\ref{thm:main_theorem}. Before moving to the proof, we record the ``folding'' property of
the Thue--Morse sequence in terms of $\eps_n$:
\begin{equation}\label{eq:folding}
\eps_{n+m} = \eps_n\eps_m
\qquad
\forall n,m\geq 0
\quad
\text{with}
\quad
n < 2^h
\end{equation}
where $h$ is the greatest nonnegative integer $\nu$ such that $2^\nu$ divides $m$. In particular,
\begin{equation}\label{eq:folding2}
\eps_{n+2^h} = -\eps_n
\qquad
\forall n,m\geq 0
\quad
\text{with}
\quad
n < 2^h
\end{equation}
(see~\cite{DekkingMendesFrancePoortenI} and~\cite[Section~2.2]{Lothaire1}).
\begin{proof}[Proof of Theorem~\ref{thm:main_theorem}]
We only deal with the case $k\geq 1$, the case $k=0$ being very similar. Also, we postpone the proof that
the sets $X_1,X_2,\dots$ are pairwise disjoint, countable and with $X_k\cap\overline \Q=\emptyset$ for
all $k$ to Section~\ref{sec:exceptional}.

First, assume that $\tau\notin Y_{k+1}$. By Theorem~\ref{thm:blocks}, there exist $N,j\in\N$
such that
\begin{equation}\label{eq:tail}
(s_n)_{n\geq N}=(\kappa_j\cdot B_{k_j},\kappa_{j+1}\cdot B_{k_{j+1}},\kappa_{j+2}\cdot B_{k_{j+2}},\dots)
\end{equation}
with $k_{i}\geq k+2$ for all $i\geq j$. By~\eqref{eq:folding}, for $h,k\geq 0$ we have
\begin{equation*}
B_{k+h}=(\eps_0\cdot B_{k+2},\eps_1\cdot B_{k+2},\dots,\eps_h \cdot B_{k+2}).
\end{equation*}
Hence,~\eqref{eq:tail} can be rewritten as
\begin{equation*}
(s_n)_{n\geq N}=(\delta_0\cdot B_{k+2},\delta_1 \cdot B_{k+2},\delta_2\cdot B_{k+2},\dots)
\end{equation*}
where $(\delta_m)_m$ is a sequence with values in $\{\pm 1\}$.

Now let $M=N+2^{k+2}m$ with $m\geq 0$, so that $s_{n+M}=\delta_m\eps_{n}$ for $0\leq n< 2^{k+2}$. In
particular, $s_{M+2^{k+1}}=-s_M$ and so by~\eqref{eq:signs_to_ineq_2}  and~\eqref{eq:thm_to_thm_b} we
have
\begin{equation*}
|\sigma_{M-1}-\tau|
\leq |\sigma_{M-1+2^{k+1}}-\sigma_{M-1}|
=    \Bigg|\sum_{0\leq \ell<2^{k+1} }\frac{\delta_m\eps_\ell}{M+1+\ell} \Bigg|
\sim \frac{c_{k+1}}{M^{k+2}},
\end{equation*}
as $m\to\infty$ (i.e., $M\to\infty$).
It follows that, writing $0\leq r<2^{k+2}$ as in Lemma~\ref{lem:thm_to_thm}, we have
\begin{equation*}
M^{k+1}(\sigma_{M-1+r}-\tau)
= M^{k+1}(\sigma_{M-1}-\tau)+M^{k+1}\sum_{0\leq \ell<r}\frac{\delta_m\eps_\ell}{M+\ell}
= o(1)+ \frac{(-1)^q\delta_mc_{h_q}}{M^{h_q-k}}(1+o(1))
\end{equation*}
as $m\to\infty$, whence
\begin{equation*}
\lim_{m\to\infty}\delta_m (N+2^{k+2}m)^{k+1}(\sigma_{N+2^{k+2}m-1+r}-\tau)
=\begin{cases}
  0   , & \text{if } r\in\{0,2^{k+1}\};\\
  c_k , & \text{if } r=2^k            ;\\
- c_k , & \text{if } r=2^k+2^{k+1}   ;\\
(-1)^q\infty , &\text{if }r\notin\{0,2^k,2^{k+1},2^k+2^{k+1}\}.
\end{cases}
\end{equation*}
The fact that $S_{k+1}'(\tau)=\{0,\pm c_k,\pm \infty\}$ then follows.
\smallskip

Now assume that $\tau\in X_{k+1}$. By Theorem~\ref{thm:blocks}, there exists $N\in\N$ such
that
\begin{equation*}
(s_{n})_{n\geq N}=(B_{k+1},B_{k+1},B_{k+1},\dots).
\end{equation*}
Since the sets $X_h$ are pairwise disjoint, we have that $\tau\notin Y_k$. In particular, the above
argument gives
\begin{align*}
\lim_{m\to\infty}(N+2^{k+1}m)^{k+1}(\sigma_{N+2^{k+1}m-1+r}-\tau)
= (-1)^q\infty
\end{align*}
for $0<r<2^{k+1}$, $r\neq 2^k$. We shall now show that the above limit is $c_k/2$ when $r=0$.
Similarly one shows that the limit is $-c_k/2$ when $r=2^k$.\\
By grouping the sums in the $B_{k+1}$ blocks, one sees that
\begin{equation*}
\tau
=\sum_{n=1}^{\infty}\frac{s_n}{n}
=\sigma_{N-1}+\sum_{h=0}^{\infty}\sum_{\ell=0}^{2^{k+1}-1}\frac{\eps_\ell}{N+h2^{k+1}+\ell}
=\sigma_{N-1}+\sum_{h=0}^{\infty}g_{k+1}(N+h2^{k+1}).
\end{equation*}
It follows  that
\begin{equation}\label{eq:case_0}
\tau-\sigma_{N+2^{k+1}m-1}=\sum_{h=m}^{\infty}g_{k+1}(N+h2^{k+1}).
\end{equation}
By Lemma~\ref{lem:thm_to_thm} $g_{k+1}$ is positive and decreasing. Thus by~\eqref{eq:thm_to_thm_b} as
$x\to+\infty$ one has
\begin{align*}
\sum_{h\geq 0}g_{k+1}(x+h2^{k+1})&=\int_0^{+\infty}g_{k+1}(x+u2^{k+1})\d u + O(x^{-k-2})\\[-0.5em]
&=\frac{1}{2^{k+1}}\int_{x}^{+\infty}g_{k+1}(u)\d u + O(x^{-k-2})\\
&=\frac{1}{2^{k+1}}\int_{x}^{+\infty}\Big(\frac{c_{k+1}}{u^{k+2}}+ o(u^{-k-2})\Big)\d u + O(x^{-k-2})\\
&=\frac{c_{k+1}}{2^{k+1}(k+1)x^{k+1}}+ o(x^{-k-1})
 =\frac{c_k}{2x^{k+1}}+ o(x^{-k-1}),
\end{align*}
since $(k+1)2^kc_k=c_{k+1}$. Thus, by~\eqref{eq:case_0} we obtain
\begin{align*}
\lim_{m\to\infty}(N+2^{k+1}m)^{k+1}(\sigma_{N+2^{k+1}m-1}-\tau) = c_k/2,
\end{align*}
as claimed. Collecting the above results one obtains $S_{k+1}'(\tau)=\{\pm c_k/2,\pm \infty\}$ for
$\tau\in X_{k+1}$.
\smallskip

Finally, the case $\tau\in Y_k$ of~\eqref{eq:main_theorem} can be proven easily along the same lines.
\end{proof}
\begin{proof}[Proof of the Corollary~\ref{cor:inequalities}]
Since $\tau\notin Y_k$, then by Theorem~\ref{thm:blocks} one has that there exists $N\in\N$ such that the
sequence $(s_{n})_{n\geq N}$ can be written as a concatenation of blocks $\pm B_{k+1}=\pm(B_k,-B_{k})$.
It follows that the sequence $(s_{n})_{n}$ contains infinitely many blocks $(B_k,-B_k)$ and infinitely
many blocks $(-B_k,B_k)$. Now if $(s_r)_{m\leq r< m+2^{k+1}}=(B_k,-B_k)$ then
applying~\eqref{eq:signs_to_ineq_2} with $n=m-1$ and $n'=m+2^k-1$, we obtain
\begin{equation*}
0\leq (\tau-\sigma_{m-1})\leq \bigg| \sum_{r=m}^{m+2^k-1} \eps_r/r\bigg|=g_{k}(m)<c_k m^{-k-1}<c_k(m-1)^{-k-1},
\end{equation*}
where in the third inequality we used Lemma~\ref{lem:thm_to_thm}. Also, by
Proposition~\ref{prop:equality}, the first inequality is strict if $m$ is large enough. It follows
that~\eqref{eq:ineq2} is satisfied for infinitely many $m$ and one proves in the same way that the same
holds for~\eqref{eq:ineq1}.
\end{proof}

\section{Thue--Morse sums and the proof of Theorem~\ref{thm:blocks}}
\label{sec:proof_blocks}
\subsection{The function $g_k(x)$ and other Thue--Morse sums}
We recall that the function $g_k(x)$ was defined in~\eqref{eq:defin_g_k} as a certain logarithmic
average of the Thue--Morse sequence, and that we defined the operator $\Delta_\alpha f$ as
$(\Delta_\alpha f)(x):= f(x) - f(x+\alpha)$.
\begin{lem}\label{lem:iteration_g_k}
For all $k\in\N$ and $x>0$ we have
\begin{equation}\label{eq:A2n}
g_k(x) = (\Delta_{2^{k-1}}\circ\dots\circ\Delta_{2}\circ\Delta_{1})g_0(x)
\end{equation}
or, equivalently,
\begin{equation}\label{eq:A3n}
g_k(x)
= (-1)^k\int_{0}^{2^{k-2}}\dots \int_{0}^{2^0} g^{(k)}_0(x+u_1+\cdots+u_k) \d u_k\cdots\d u_2\d u_1.
\end{equation}
\end{lem}
\begin{proof}
Since $g_0$ admits derivatives of any order for $x>0$, Equation~\eqref{eq:A3n} follows immediately
from~\eqref{eq:A2n}. Now we prove~\eqref{eq:A2n} by induction on $k$. For $k=1$, the equality is
immediate. Now let $k\geq 2$ and suppose the claim is true for $k-1$. Then we split the range of
summation $[0,2^k)$ in the definition of $g_k$ into $[0,2^{k-1}) \cup [2^{k-1},2^k)$ and shift the variable
in the second range by $2^{k-1}$. We get
\begin{align*}
g_k(x)
&= \sum_{0\leq \ell< 2^k}\frac{\eps_\ell}{x+\ell}
 = \sum_{0\leq \ell< 2^{k-1}}\frac{\eps_\ell}{x+\ell} + \sum_{0\leq \ell< 2^{k-1}}\frac{\eps_{\ell+2^{k-1}}}{x+2^{k-1}+\ell}
 = g_{k-1}(x)-g_{k-1}(x+2^{k-1}),
\end{align*}
since $\eps_{\ell+2^{h-1}} = -\eps_\ell$ by the folding property~\eqref{eq:folding2}. Thus, by the
inductive hypothesis we have
\begin{equation*}
g_k(x)=\Delta_{2^{k-1}} g_{k-1}(x)= (\Delta_{2^{k-1}}\circ\dots\circ\Delta_{2}\circ\Delta_{1})g_0(x),
\end{equation*}
as claimed.
\end{proof}
The following lemma shows that any logarithmic average of the Thue--Morse sequence can be expressed as a
combination of the functions $g_k(x)$.
\begin{lem}\label{lem:A1n}
Let $r := \sum_{0\leq j\leq q}2^{h_j} > 0$ with $h_0>h_1>\dots> h_q$, and let $r_j:=\sum_{i<j}2^{h_i}$.
Then
\begin{equation*}
\sum_{0\leq \ell< r}\frac{\eps_\ell}{x+\ell} = \sum_{0\leq j\leq q}(-1)^{j}g_{h_j}(x+r_j).
\end{equation*}
\end{lem}
\begin{proof}
The proof is by induction on $q$. The case $q=0$ holds true by definition. Let $q\geq 1$, and suppose the
claim is true for $q-1$. The definition of $r$ shows that $r=r_q + 2^{h_q}$, so that, splitting the range
of summation $[0,r)$ into $[0,r_q)\cup[r_q,r_q+2^{h_q})$ and shifting the variable in the second range by
$r_q$, we get
\begin{equation*}
\sum_{0\leq \ell< r}\frac{\eps_\ell}{x+\ell}
= \sum_{0\leq \ell< r_q}\frac{\eps_\ell}{x+\ell} + \sum_{0\leq \ell< 2^{h_q}}\frac{\eps_{\ell+r_q}}{x+r_q+\ell}.
\end{equation*}
The dyadic representation of $r_q$ contains $q-1$ nonzero digits, hence the inductive hypothesis may be
applied to the first term. Moreover, in the second sum $\ell<2^{h_q}< 2^{h_{q-1}}\leq r_q$, thus $\eps_{\ell+r_q}
= \eps_\ell\eps_{r_q} = (-1)^q\eps_\ell$, so that
\begin{equation*}
\sum_{0\leq \ell< r}\frac{\eps_\ell}{x+\ell}
= \sum_{0\leq j\leq q-1}(-1)^{j}g_{h_j}(x+r_j) + (-1)^q\sum_{0\leq \ell< 2^{h_q}}\frac{\eps_\ell}{x+r_q+\ell}.
\end{equation*}
By definition the second sum on the right is $g_{h_q}(x+r_q)$ and the proof is complete.
\end{proof}
We now describe the asymptotic behavior of $g_k(x)$.
\begin{lem}
Let $k\geq 0$. Then, as $x\to\infty$, we have $g_k(x)= \frac{c_k}{x^{k+1}}+O_k(x^{-k-2})$, where $c_k$ is
as in Theorem~\ref{thm:main_theorem}.
\end{lem}
\begin{proof}
We show that $g_k(x)=c_k/x^{k+1}+\phi_k(1/x)$ with $\phi_k$ analytic in a neighborhood of $0$ and with a
zero of order at least $ k+2$ at $x=0$. This is obvious for $k=0$. Suppose it is true for $k\geq 0$. We
have
\begin{equation*}
g_{k+1}(x)
= \frac{c_k}{x^{k+1}}+\phi_k(1/x)-\frac{c_k}{(x+2^k)^{k+1}}-\phi_k(1/(x+2^k))
= \frac{(k+1) 2^k c_k}{x^{k+2} }+\phi_{k+1}(1/x)
\end{equation*}
with
\begin{equation*}
\phi_{k+1}(x)=\phi_k(x)-\phi_k(x(1+ 2^k x)^{-1}) - c_k x^{k+1} \left((1+ 2^k x)^{-k-1} - 1 + (k + 1)2^k x\right).
\end{equation*}
Clearly $\phi_{k+1}(x)$ is analytic in a neighborhood of $0$ and has a zero of order at least $ k+3$ at
$x=0$. Moreover, by definition, $(k+1) 2^k c_k=2^{\frac{k(k-1)}2+k} (k+1)!=c_{k+1}$, as desired.
\end{proof}
\begin{cor}\label{cor:asymp}
With the same notation as in Lemma~\ref{lem:A1n}, as $x\to\infty$ we have
\begin{equation*}
\sum_{0\leq \ell< r}\frac{\eps_\ell}{x+\ell} \sim (-1)^{q}\frac{c_{h_q}}{x^{h_q+1}}.
\end{equation*}
\end{cor}

\subsection{Inequalities for $g_k(x)$}
\begin{lem}\label{lem:g_signs}
For all $m,k\geq 0$ and all $x>0$ we have
\begin{equation*}
2^{\binom{k}2}\frac{(m+k)!}{(x+2^k)^{m+k+1}}<(-1)^mg^{(m)}_k(x)<2^{\binom{k}2}\frac{(m+k)!}{x^{m+k+1}}.
\end{equation*}
In particular, for all $k\geq 0$, $g_k(x)$ is positive and decreasing for $x>0$.
\end{lem}
\begin{proof}
The claim is clear for $k=0$. Suppose $k\geq 1$. The operator $\Delta_\alpha$ commutes with the
differentiation. Thus~\eqref{eq:A2n} and~\eqref{eq:A3n} applied to $g_k^{(m)}$ show that
\begin{equation*}
(-1)^m g^{(m)}_k(x)
= \int_{0}^{2^{k-1}}\int_{0}^{2^{k-2}}\dots \int_{0}^{2^0} (-1)^{m+k}g^{(m+k)}_0(x+u_1+\cdots+u_k) \d u_k\cdots\d u_2\d u_1.
\end{equation*}
The measure of the domain of integration is $2^{\binom{k}2}$ and so the claim follows, since for $x< u<
x+2^{k-1}+\cdots+2^0=x+2^k-1$ one has
\begin{equation*}
\frac{(m+k)!}{(x+2^k)^{m+k+1}}
< (-1)^{m+k}g^{(m+k)}_0(u)
= \frac{(m+k)!}{u^{m+k+1}}
< \frac{(m+k)!}{x^{m+k+1}}.\qedhere
\end{equation*}
\end{proof}
\begin{cor}\label{cor:A2n}
For all $k\geq 0$ and all $x>0$ we have
\begin{equation*}
g_k(x)<g_{k-1}(x).
\end{equation*}
\end{cor}
\begin{proof}
We have $g_k(x)=g_{k-1}(x)-g_{k-1}(x+2^{k-1})<g_{k-1}(x)$.
\end{proof}
\begin{lem}\label{lem:A5n}
Let $k\geq 0$. Then for all $x\geq (k+1)2^{k+2}$ we have $g_k(x)< \frac{4}{3}g_k(x+2^k)$.
\end{lem}
\begin{proof}
For $a,k\geq 0$, let
\begin{equation*}
h_k(x,a):=4g_k(x+a)-3g_k(x).
\end{equation*}
The operator $\Delta_\alpha$ is linear, and thus for all $k\geq 1$ we have
\begin{equation*}
h_k(x,a)=h_{k-1}(x,a)-h_{k-1}(x+2^{k-1},a).
\end{equation*}
In particular, the analogues of the identities~\eqref{eq:A2n} and~\eqref{eq:A3n} also hold for $h_k$.
Moreover, for $m\geq 0$ and $a/x<\sqrt[m+1]{4/3}-1$ we have
\begin{equation*}
(-1)^mh^{(m)}_0(x,a)
= m!\bigg(\frac4{(x+a)^{m+1}}-\frac3{x^{m+1}}\bigg)
> 0.
\end{equation*}
Reasoning as in the proof of Lemma~\ref{lem:g_signs} we get that
\begin{equation*}
h_k(x,a)>0
\end{equation*}
for $a/x<\sqrt[k+1]{4/3}-1$. Now for $x\geq 0$ and $0\leq \rho\leq 1$ we have $(1+x)^{\rho}-1\geq \rho
x+\rho(\rho-1)\frac {x^2}2$ and so $\sqrt[k+1]{4/3}-1\geq \frac{5k+6}{18(k+1)^2}>\frac{1}{4(k+1)}$. The
lemma then follows by taking $a=2^k$.
\end{proof}
\begin{lem}\label{lem:A6n}
For $x\geq 2^{k+1}k$ and $0\leq h<k$ we have $g_k(x)<\frac{1}{2}g_h(x+2^k)$.
\end{lem}
\begin{proof}
By Corollary~\ref{cor:A2n}, it suffices to prove the claim when $h=k-1$.\\
Applying Lemma~\ref{lem:A5n} twice, for $x\geq 2^{k+1}k$ we have
\begin{align*}
g_k(x)-\tfrac{1}{2} g_{k-1}(x+2^k)
& = g_{k-1}(x)-g_{k-1}(x+2^{k-1})-\tfrac{1}{2} g_{k-1}(x+2^k)\\
&< \tfrac{1}{3} g_{k-1}(x+2^{k-1})-\tfrac{1}{2} g_{k-1}(x+2^k)
< -\tfrac{1}{18} g_{k-1}(x+2^k)
< 0,
\end{align*}
as desired.
\end{proof}
\begin{lem}\label{lem:A7n}
Let $r = \sum_{0\leq j\leq q}2^{h_j}>0$ with $h_0>h_1>\dots> h_q$. Then, if $r<2^k$ and $x\geq
2^{k+1}(k+1)$, we have
\begin{equation*}
(-1)^q\sum_{0\leq\ell< r}\frac{\eps_\ell}{x+\ell}
> g_k(x-2^k)
> g_k(x)
> 0.
\end{equation*}
\end{lem}
\begin{proof}
First we observe that the inequality $g_k(x-2^k)>g_k(x)>0$ follows by Lemma~\ref{lem:g_signs}.\\
Write $r_j=\sum_{i<j}2^{h_i}$; by Lemma~\ref{lem:A1n} we have
\begin{equation*}
(-1)^q\sum_{0\leq \ell< r}\frac{\eps_\ell}{x+\ell} - g_k(x-2^k)
= \sum_{0\leq j\leq q}(-1)^{j+q}g_{h_j}(x+r_j) - g_k(x-2^k).
\end{equation*}
If $q$ is odd this is
\begin{multline*}
\big(g_{h_q}(x+r_q)-g_{h_{q-1}}(x+r_{q-1})\big)
+\cdots
+\big(g_{h_{3}}(x+r_{3})-g_{h_{2}}(x+r_{2})\big)\\
+\big(g_{h_{1}}(x+r_{1})-g_{h_0}(x+r_0)-g_k(x-2^k)\big).
\end{multline*}
For all $\ell\geq 1$ we have $r_\ell=r_{\ell-1}+2^{h_{\ell-1}}$ and thus, by Lemma~\ref{lem:A6n}, for all
$x\geq 2^{k+1}k$ we have
\begin{equation*}
 g_{h_\ell}(x+r_\ell)-g_{h_{\ell-1}}(x+r_{\ell-1})>\tfrac{1}{2}g_{h_\ell}(x+r_\ell)>0.
\end{equation*}
Moreover, since $r_0=0$ and $h_0<k$, by Lemma~\ref{lem:A6n} for $x>2^{k+1}(k+1)>2^{k+1}k+2^k$ we have
\begin{equation*}
 g_{h_{1}}(x+r_{1})-g_{h_0}(x+r_0)-g_k(x-2^k)> g_{h_{1}}(x+r_{1})-\tfrac32g_{h_0}(x+r_0)>0.
\end{equation*}
The case $q$ even is analogous and slightly simpler.
\end{proof}
\begin{cor}\label{cor:A3n}
Let $0<r<2^k$ and $x \geq 2^{k+1}(k+1)$. Then
\begin{equation*}
\sgn\bigg(\sum_{0\leq \ell< r}\frac{\eps_\ell}{x+\ell}\bigg) = -\eps_r.
\end{equation*}
\end{cor}
\begin{proof}
Lemma~\ref{lem:A7n} shows that the sign is $(-1)^q$, which is $-\eps_r$ because $q+1$ is the number of
times the digit $1$ appears in the dyadic expansion of $r$.
\end{proof}

\subsection{Proof of Theorem~\ref{thm:blocks}}
The following lemma provides the crucial step in the proof of Theorem~\ref{thm:blocks}. It shows, for
large enough $n$, that if the distance of $\sigma_{n-1}$ to $\tau$ is less than $g_k(n)$  then either
this distance is also less than $g_{k+1}(n)$ or $\sigma_{n-1+2^k}$ has distance from $\tau$ which is less
than $g_k(n+2^k)$.
\begin{lem}\label{lem:A8n}
Let $k\geq 0$ and assume that $n\geq 2^{k+1}(k+1)$ is such that
\begin{equation}\label{eq:A8n}
\sigma_{n-1}\leq \tau <\sigma_{n-1}+g_k(n).
\end{equation}
Then for $0\leq r< 2^{k+1}$ one has $s_{n+r}=\eps_r$ and one of the following inequalities holds:
\begin{align}
&\sigma_{n-1}\leq \tau <\sigma_{n-1}+g_{k+1}(n),          \label{eq:A9n}\\
&\sigma_{n-1+2^k}-g_k(n+2^k)\leq \tau <\sigma_{n-1+2^k}.  \label{eq:A10n}
\end{align}
Similarly, if $n\geq 2^{k+1}(k+1)$ is such that
\begin{equation}\label{eq:A11n}
\sigma_{n-1}-g_k(n)\leq \tau <\sigma_{n-1},
\end{equation}
then for $0\leq r< 2^{k+1}$ one has $s_{n+r} = -\eps_r$ and one of the following inequalities holds:
\begin{align*}
&\sigma_{n-1}-g_{k+1}(n)\leq \tau <\sigma_{n-1},        \\
&\sigma_{n-1+2^k}\leq \tau <\sigma_{n-1+2^k}+g_k(n+2^k).
\end{align*}
\end{lem}
\begin{proof}
We shall consider only the case where~\eqref{eq:A8n} holds, the other case being analogous.\\
Assuming~\eqref{eq:A8n}, by Lemma~\ref{lem:A7n} and Corollary~\ref{cor:A3n} we have
\begin{equation}\label{eq:A14n}
\sgn\bigg(\sigma_{n-1}-\tau+\sum_{0\leq \ell< r}\frac{\eps_\ell}{n+\ell}\bigg)
= \sgn\bigg(\sum_{0\leq \ell< r}\frac{\eps_\ell}{n+\ell}\bigg)
= -\eps_r
\end{equation}
for all $r \in (0,2^k)$. Moreover, if $\sigma_{n-1}\neq\tau$, then the claim holds also for $r=0$, since
$\sgn(\sigma_{n-1}-\tau)=-1=-\eps_0$. The equalities in~\eqref{eq:A14n} and the definition of the greedy
algorithm then imply that, in any case,
\begin{equation}\label{eq:A15n}
\sigma_{n-1+r} = \sigma_{n-1}+\sum_{0\leq \ell< r}\frac{\eps_\ell}{n+\ell}
\quad
\text{for all $r \in [0,2^k]$}.
\end{equation}
This relation with $r=2^k$ and the equality $\sum_{0\leq\ell< 2^k}\frac{\eps_\ell}{n+\ell}=g_k(n)$ in
Lemma~\ref{lem:A1n} show that
\begin{equation*}
\sigma_{n-1+2^k} = \sigma_{n-1}+\sum_{0\leq \ell< 2^k}\frac{\eps_\ell}{n+\ell} = \sigma_{n-1}+g_k(n),
\end{equation*}
which is part of what we have to prove in this case. Further, this equality and~\eqref{eq:A8n} imply that
\begin{equation}\label{eq:A16n}
\sigma_{n-1+2^k}-g_k(n) \leq \tau <\sigma_{n-1+2^k}.
\end{equation}
In particular, this implies that
\begin{equation}\label{eq:A17n}
\sigma_{n+2^k}
= \sigma_{n-1+2^k}-\frac{1}{n+2^k}
= \sigma_{n-1+2^k}-\frac{\eps_0}{n+2^k}.
\end{equation}
By~\eqref{eq:A16n} and appealing once again to Lemma~\ref{lem:A7n} and Corollary~\ref{cor:A3n}, we have
\begin{equation*}
\sgn\bigg(\sigma_{n-1+2^k}-\tau-\sum_{0\leq \ell< r}\frac{\eps_\ell}{n+2^k+\ell}\bigg)
= -\sgn\bigg(\sum_{0\leq \ell< r}\frac{\eps_\ell}{n+2^k+\ell}\bigg)
= \eps_r
\end{equation*}
for all $r \in (0,2^k)$. Thus, since we have also~\eqref{eq:A17n} we must have
\begin{equation}\label{eq:A18n}
\sigma_{n-1+2^k+r} = \sigma_{n-1+2^k} - \sum_{0\leq \ell< r}\frac{\eps_\ell}{n+2^k+\ell}
\quad
\text{for all $r \in [0,2^k]$}.
\end{equation}
By~\eqref{eq:A15n}, \eqref{eq:A17n} and~\eqref{eq:A18n} we have
\begin{equation*}
\sigma_{n-1+r} = \sigma_{n-1} + \sum_{0\leq \ell< r}\frac{\eps_\ell}{n+\ell}
\quad
\text{for all $r \in [0, 2^{k+1}]$}.
\end{equation*}
This proves that $s_{n+r}=\eps_r$ for $r=0,\ldots,2^{k+1}-1$. Moreover, the case $r=2^{k+1}$ yields
\begin{equation}\label{eq:A20n}
\sigma_{n-1+2^{k+1}}
= \sigma_{n-1}+\sum_{0\leq \ell< 2^{k+1}}\frac{\eps_\ell}{n+\ell}
= \sigma_{n-1}+g_{k+1}(n).
\end{equation}
Then we have two possibilities: either $\sigma_{n-1+2^{k+1}}>\tau$ or $\sigma_{n-1+2^{k+1}}\leq \tau$.
In the first case, by~\eqref{eq:A8n} we have~\eqref{eq:A9n}. In the second case, we observe that since
$\tau <\sigma_{n-1+2^k}$ by~\eqref{eq:A16n}, then comparing~\eqref{eq:A15n} with $r=2^k$
and~\eqref{eq:A20n}, we get
\begin{equation*}
\sigma_{n-1+2^k}-g_k(n+2^k)=\sigma_{n-1+2^{k+1}}\leq \tau <\sigma_{n-1+2^k},
\end{equation*}
so~\eqref{eq:A10n} holds.
\end{proof}
\begin{cor}\label{cor:A4n}
Let $k\geq 0$ and let $n\geq 2^{k+1}(k+1)$.
Suppose that we have $(s_m)_{n\leq m<n+2^k}=B_k$ or $(s_n)_{n\leq m<n+2^k}=-B_k$ with $B_k$ as in
Theorem~\ref{thm:blocks}. Then there exists a sequence $(\delta_i)_{i\geq 0}$ with $\delta_i\in\{\pm 1\}$
for all $i$ such that
\begin{equation*}
(s_m)_{m\geq n}=(\delta_0\cdot B_k,\delta_1\cdot B_k,\delta_2\cdot B_k,\dots).
\end{equation*}
\end{cor}
\begin{proof}
Suppose  $(s_m)_{n\leq m<n+2^k}=B_k$, the other case is proved in a similar way. Notice that it suffices
to show that $(s_m)_{n+2^k\leq m<n+2^{k+1}}=\pm B_k$, since one can then iterate the same argument.\\
Since $s_n=\eps_0=1$, we have $\sigma_{n-1}\leq \tau$. Moreover, since $s_{n+{2^{k-1}}} = \eps_{2^{k-1}}
=-1$, by~\eqref{eq:signs_to_ineq_2} and the definition of $g_{k-1}$ (which is positive by
Lemma~\ref{lem:g_signs}), we deduce that
\begin{equation*}
\sigma_{n-1}
\leq \tau
< \sigma_{n-1} +g_{k-1}(n)
= \sigma_{n+2^{k-1}-1},
\end{equation*}
where the second inequality is strict by~\eqref{eq:signs_to_ineq}.
Thus the hypothesis~\eqref{eq:A8n} in Lemma~\ref{lem:A8n} is satisfied with $k-1$ in place of $k$, whence
either~\eqref{eq:A9n} or~\eqref{eq:A10n} holds with $k$ in place of $k+1$. In the first case, we have
that~\eqref{eq:A8n} is satisfied; hence, applying Lemma~\ref{lem:A8n} once again, we obtain
$s_{n+r}=\eps_r$ for $0\leq r< 2^{k+1}$, i.e., $(s_m)_{n\leq m<n+2^{k+1}}=B_{k+1}=(B_k,-B_k)$, and in
particular $(s_m)_{n + 2^k \leq m < n + 2^{k+1}} = -B_k$. In the second case, we notice
that~\eqref{eq:A10n}, with $k-1$ in place of $k$, is actually hypothesis~\eqref{eq:A11n} with $k-1$ in
place of $k$ and $n+2^k$ in place of $n$. Thus, applying Lemma~\ref{lem:A8n} we obtain $(s_m)_{n+2^k\leq
m<n+2^{k+1}}=B_k$, as desired.
\end{proof}
We are now in a position to prove Theorem~\ref{thm:blocks}. First, we define the exceptional sets $X_k$.
Also we recall that $Y_k:=\cup_{h\leq k}X_h$ with $Y_0:=\emptyset$.
\begin{defin}
For all $k\in\N$ the set $X_k$ is defined as
\begin{equation}\label{eq:def_x_k}
X_k := \Big\{\tau \in\R \colon \textup{$\exists N\geq 0$ and $0\leq m<2^k$ s.~t. }
             s_n = f_k(n+m)\ \forall n\geq N\Big\},
\end{equation}
where we recall that $f_{k}$ denotes the $2^k$-periodic function such that $f_{k}(n)=\eps_{n}$ for all $n \in [0,2^k)$.
\end{defin}
\begin{proof}[Proof of Theorem~\ref{thm:blocks}]
First, we observe that the result is tautological for $\tau\in X_k$. Indeed, by the definition of $X_k$
there exists $N$ such that
\begin{equation*}
(s_n)_{n\geq N}=(s_1,s_2,\dots,s_{N-1}, B_k, B_k,B_k,\dots)
=(s_1\cdot B_0,s_2\cdot B_0,\dots,s_{N-1}\cdot B_0, B_k, B_k,B_k,\dots).
\end{equation*}
Thus, let us assume that $\tau\notin Y_\infty$.

By Proposition~\ref{prop:A1n} $(s_n)_n$ has infinitely many changes of signs. Equivalently, it contains
infinitely many blocks $B_1$. In particular, we can find $N_1$  such that $N_1\geq 2^{1+1}\cdot (1+1)=8$
and $(s_{m})_{N_1\leq m<N_1+2}=B_1$. Thus, by Corollary~\ref{cor:A4n} there exists a sequence
$(\delta_{i,1})_{i\geq 0}$ with values in $\{\pm 1\}$ such that
\begin{equation}\label{eq:iterative_1}
(s_n)_{n\geq N_1}=(\delta_{0,1}\cdot B_1,\delta_{1,1}\cdot B_1,\delta_{2,1}\cdot B_1,\dots).
\end{equation}
The sequence $(\delta_{i,1})_{i\geq 0}$ contains infinitely many changes of signs. Indeed, it cannot be
equal to $1$ for all $i$ large enough, since otherwise we would have $\tau\in X_1$, and for the same
reason it cannot be equal to $-1$ for all $n$ large, since
\begin{equation*}
(-B_1,-B_1,-B_1,\dots)=(-B_0,B_1,B_1,B_1,\dots).
\end{equation*}
Thus, we can find $i$ such that $\delta_i=-\delta_{i+1}=1$ and such that the corresponding index $n$ on
the left hand side of~\eqref{eq:iterative_1} is $N_2\geq 2^{2+1}\cdot (2+1)=24$ (and $N_2> N_1$). We then
have $(s_{m})_{N_2\leq m<N_2+2^2}=(B_1,-B_1)=B_2$ and so by Corollary~\ref{cor:A4n} there exists a
sequence $(\delta_{i,2})_{i\geq 0}$ with values in $\{\pm 1\}$ such that
\begin{equation*}
(s_n)_{n\geq N_2}=(\delta_{0,2}\cdot B_2,\delta_{1,2}\cdot B_2,\delta_{2,2}\cdot B_2,\dots).
\end{equation*}
As before, one sees that $(\delta_{i,2})_{i\geq 0}$ contains infinitely many changes of signs. We can then
keep iterating this process, whence obtaining sequences $(k_i)_i$ and $(\kappa_i)_i$ with the desired
properties.
\end{proof}

\section{The exceptional sets $X_k$}\label{sec:exceptional}
In this section we study the exceptional sets $X_k$, defined in~\eqref{eq:def_x_k}.

As observed in the introduction, the elements in $X_k$ can be written as a rational number plus $U_{k,m}$
for some $m \in [0,2^k)$, where the constant $U_{k,m}$ is as defined in~\eqref{eq:def_U}. We also recall
that $U_{k,m}=-U_{k,m+2^{k-1}}$ for all $m \in [0,2^{k-1})$. The values of $U_{k,m}$ for $k=1,2$ have been
given in~\eqref{eq:u01}. For $k=3$ we have
\begin{align*}
&  U_{3,0}=\frac{4\sqrt{2}\log(2 - \sqrt{2})-\pi - (\sqrt{2}+1)\log 4}{8},
&& U_{3,1}=\frac{(1-2\sqrt{2})\pi+\log 4}{8},                             \\
&  U_{3,3}=\frac{4\sqrt{2}\log(2 - \sqrt{2})+\pi - (\sqrt{2}+1)\log 4}{8},
&& U_{3,2}=\frac{(2\sqrt{2}-1)\pi+\log 4}{8}.
\end{align*}
We now show that $U_{k,m}$ can be written as a linear combination of logarithms. Baker's theorem will
then yield the transcendence of $U_{k,m}$.
\begin{prop}\label{prop:A2n}
For every $k\geq 1$, $m\in\Z$, we have
\begin{equation*}
U_{k,m}
= \sum_{\substack{a=1\\ a\odd}}^{2^k-1}c(a)e^{2\pi ia m/2^k} \log\big(1-e^{2\pi ia/2^k}\big),
\qquad
c(a):= i^k e^{\pi i a/2^k} \prod_{j=1}^k \sin\Big(\frac{\pi a}{2^j}\Big).
\end{equation*}
\end{prop}
\begin{proof}
We write $f_k(n)$ in terms of additive characters:
\begin{equation}\label{eq:fourier}
f_k(n)
= \frac{1}{2^k} \sum_{a=0}^{2^k-1}\Bigg(\sum_{\ell=0}^{2^k-1}\eps_\ell\, e^{-2\pi i \ell a/2^k} \bigg)e^{2\pi i an/2^k}.
\end{equation}
The sum in brackets is equal to $P_k(e^{-2\pi i a/2^k})$, where $P_k(x):=\sum_{\ell=0}^{2^k-1}\eps_\ell
x^\ell$. The uniqueness of the base-$2$ expansion gives the factorization $P_k(x) =
\prod_{j=0}^{k-1}(1-x^{2^j})$, so that
\begin{align*}
\frac{1}{2^k}\sum_{\ell=0}^{2^k-1}\eps_\ell\, e^{-2\pi i \ell a/2^k}
&= \frac{1}{2^k}\prod_{j=0}^{k-1}\big(1-e^{-2\pi i a/2^{k-j}} \big)
 = i^k \bigg(\prod_{j=0}^{k-1}e^{-\pi i a2^j/2^{k}} \bigg)\prod_{j=1}^k \sin\Big(\frac{\pi a}{2^j}\Big)\\
&= i^k e^{-\pi i a(2^{k}-1)/2^k} \prod_{j=1}^k \sin\Big(\frac{\pi a}{2^j}\Big)=-c(a),
\end{align*}
where in the last step we used that the product on the second line is $0$ unless $a$ is odd.

Inserting~\eqref{eq:fourier} in the definition~\eqref{eq:def_U} of $U_{k,m}$ and  exchanging the order
of summation of the two sums, a step which can be easily justified, we obtain
\begin{equation*}
U_{k,m}=-\sum_{a=1\atop a\odd}^{2^k-1}c(a)e^{2\pi i am/2^k} \sum_{n=1}^{+\infty}\frac{e^{2\pi i an/2^k}}{n}
= \sum_{a=1\atop a\odd}^{2^k-1}c(a)e^{2\pi i am/2^k} \log(1-e^{2\pi i a/2^k}),
\end{equation*}
as claimed.
\end{proof}
\begin{prop}\label{prop:A3n}
The number $U_{k,m}$ is transcendental, for every $k\geq 1$ and every $m \in [0,2^k)$.
\end{prop}
\begin{proof}
The formula in Proposition~\ref{prop:A2n} shows that $U_{k,m}$ is a non-zero linear combination with
algebraic coefficients of logarithms of the numbers $1-\zeta^a$, where $\zeta:=\exp(2\pi i/2^k)$ for
$a=1,\ldots,2^k-1$ odd. For any odd $a$, let $w_a := \zeta^{(1-a)/2}(1-\zeta^a)/(1-\zeta)$, which is well
defined because $(1-a)/2$ is an integer. Notice that $w_a$ is a positive real number for $1\leq
a<2^{k-1}$, and that $w_{2^k-a} = -w_a$; also, $w_1=1$. It follows that the sum of the $\log(1-\zeta^a)$
is also a linear combination with algebraic coefficients of numbers
\begin{equation}\label{eq:indip_num}
\log(1-\zeta),
\qquad
i\pi,
\qquad
\big(\log(w_a)\big)_{a=3,\odd}^{2^{k-1}}.
\end{equation}
The formula also shows that at least one of the coefficients is not zero; for example, the coefficient
of $\log(1-\zeta)$ is $\sum_{a=1}^{2^k-1}c(a)e^{2\pi ia m/2^k}=-\eps_m$. By Baker's theorem on linear
forms in logarithms, the transcendence of $U_{k,m}$ then follows if we can prove that the
numbers~\eqref{eq:indip_num} are $\Q$-linearly independent. Thus suppose we have a linear combination
with integer coefficients producing zero:
\begin{equation}\label{eq:A28n}
\alpha \log(1-\zeta) + \beta i\pi +\sum_{\substack{a=3\\ a\odd}}^{2^{k-1}-1} \gamma_a\log(w_a) = 0.
\end{equation}
We need to show that all the coefficients are zero. Exponentiating this identity yields
\begin{equation}\label{eq:A30n}
(1-\zeta)^\alpha(-1)^\beta\prod_{\substack{a=3\\ a\odd}}^{2^{k-1}-1} w_a^{\gamma_a} = 1.
\end{equation}
In the cyclotomic field $K:=\Q[\zeta]$ of $2^k$ roots of unity, the norm of $1-\zeta$ is equal to $2$,
whereas the norm of each of the $w_a$ is equal to $1$. Thus the previous identity implies that
$\alpha=0$. Taking the imaginary part of~\eqref{eq:A28n} then gives $\beta=0$. As a
consequence~\eqref{eq:A30n} becomes
\begin{equation}\label{eq:A31n}
\prod_{\substack{a=3\\ a\odd}}^{2^{k-1}-1} w_a^{\gamma_a} = 1.
\end{equation}
It is known that the numbers $(w_a)_{\substack{a=3,\, a\odd}}^{2^{k-1}-1}$ are cyclotomic units
generating a subgroup having finite index in the free part of the group of units $U_K$ in $O_K$~\cite[Lemma 8.1]{Washington1}. By Dirichlet's theorem the dimension (as $\Z$-module) of $U_K$ is
$2^{k-2}-1$, and this is also the number of $w_a$ appearing in~\eqref{eq:A31n}. Hence they are
multiplicatively independent and so all the $\gamma_a$ have to be equal to zero, as desired.
\end{proof}
We will now show that $U_{k,m}\in X_k$ for all $k,m$. First, we need the following lemma.
\begin{lem}\label{lem:UX}
For  $k,m\geq 0$ and $x>0$, let
\begin{equation*}
g_{k,m}(x) := \sum_{0\leq n< 2^k} \frac{f_k(n+m)}{n+x}.
\end{equation*}
Then $f_k(m)g_{k,m}(x)\geq g_{k,0}(x)=g_k(x)$ for $x>0$. In particular,
$f_k(m)g_{k,m}(x)> 0$ for $x>0$.
\end{lem}
\begin{proof}
For $k\leq 1$ the result is obvious by definition, so assume that $k\geq2$. Let
\begin{equation*}
\V_{k,m}(r):=\sum_{0\leq n\leq r}f_k(n+m)
\end{equation*}
and notice that since $\V_{k,m}(2^{k}-1)=0$, one has that $\V_{k,m}(r+2^k)=\V_{k,m}(r)$ for all $r\geq 0$.
Moreover, one easily verifies that
\begin{align*}
\V_{k,0}(r)
= \begin{cases}
f_k(r) , & r\text{ even};\\
     0  , & r\text{ odd};
\end{cases}
\end{align*}
(cf.~Proposition~\ref{prop:A4n} below). Thus, for all odd $m$ we have
\begin{align*}
f_k(m)\V_{k,m}(r)
&= f_k(m)(\V_{k,0}(r+m)-\V_{k,0}(m-1))
 = f_k(m)\V_{k,0}(r+m)+1\\
&= \begin{cases}
    1                , & r\text{ even};\\
    f_k(m)f_k(r+m)+1 , & r\text{ odd};
   \end{cases}
\end{align*}
where in the second step we used that $f_k(m-1)f_k(m)=-1$ for $m$ odd. In particular if $m$ is odd, then
\begin{align}\label{eq:fa}
f_k(m)\V_{k,m}(r)-\V_{k,0}(r)\geq 0.
\end{align}
To conclude, we observe that, if  $m=2^\nu m'< 2^k$ with $m'$ odd (or if $m=m'=0$, $\nu<k$), then
\begin{equation*}
g_{k,m}(x)
= \sum_{n=0}^{2^k-1} \frac{f_k(n+m)}{n+x}
= \sum_{\ell=0}^{2^\nu-1}\sum_{r=0}^{2^{k-\nu}-1} \frac{f_k(\ell+2^{\nu}(r+m'))}{\ell+2^\nu r+x}
= \sum_{r=0}^{2^{k-\nu}-1}f_{k-\nu}(r+m')g_{\nu}(2^\nu r+x)
\end{equation*}
since~\eqref{eq:folding} implies that $f_k(\ell+2^\nu s) = \eps(\ell)f_{k}(2^{\nu}s) =
\eps(\ell)f_{k-\nu}(s)$ for $0\leq \ell<2^\nu< 2^k$. Thus, by Abel's summation formula,
\begin{align*}
f_k(m)g_{k,m}(x)-g_{k,0}(x)
&= \sum_{0\leq r<  2^{k-\nu}}\big(f_{k-\nu}(m')f_{k-\nu}(r+m')-f_{k-\nu}(r)\big)g_{\nu}(2^\nu r+x)\\
&= -2^\nu\int_{0}^{2^{k-\nu}-1}\big(f_{k-\nu}(m')\V_{k-\nu,m'}(y)-\V_{k-\nu,0}(y)\big)g'_{\nu}(2^\nu y+x)\,{\mathrm d}y
\end{align*}
and the result follows by~\eqref{eq:fa} and Lemma~\ref{lem:g_signs}.
\end{proof}
\begin{cor}\label{cor:u}
For all $k\geq 1$ and $m \in [1,2^k)$ the identity $U_{k,m} = \sum_{n=1}^\infty f_k(n+m)/n$ gives the greedy
representation for $U_{k,m}$, i.e., $s_n(U_{k,m}) = f_k(n+m)$ for all $n$. In particular, $U_{k,m}\in
X_k$.
\end{cor}
\begin{proof}
By the definition of $U_{k,m}$, for all $N\geq 0$ we have
\begin{align*}
U_{k,m}
&= \sum_{n=1}^{N} \frac{f_k(n+m)}{n}
  + \sum_{n=N+1}^{+\infty} \frac{f_k(n+m)}{n}
 = \sum_{n=1}^{N} \frac{f_k(n+m)}{n}
  + \sum_{s=0}^{+\infty} g_{k,m+N+1}(N+1+s2^k).
\end{align*}
Lemma~\ref{lem:UX} then shows that
\begin{equation}\label{eq:B1}
f_k(m+N+1)\Big(U_{k,m} - \sum_{1\leq n\leq N} \frac{f_k(n+m)}{n}\Big)
> 0
\end{equation}
for every $N$. This fact gives the claim by induction. Indeed, this inequality with $N=0$ shows that
\begin{equation*}
f_k(m+1)>0
\iff
U_{k,m}>0
\iff
s_1(U_{k,m})>0,
\end{equation*}
proving that $f_k(m+1)= s_1(U_{k,m})$. Moreover, if the claim holds for all $n\leq N$, then
$\sum_{n=1}^N \frac{f_k(n+m)}{n} = \sigma_N(U_{k,m})$ and~\eqref{eq:B1} gives
\begin{equation*}
f_k(m+N+1)(U_{k,m} - \sigma_{N}(U_{k,m})) \geq 0,
\end{equation*}
and so $s_{N+1}(U_{k,m}) = f_k(N+1+m)$, as desired.
\end{proof}
We record some properties of the sets $X_{k}$  in the following proposition.
\begin{prop}\label{prop:X}
For all $k\geq 1$ we have $X_k\cap\overline\Q=\emptyset$.
Moreover, the sets $X_1,X_2,\dots$ are non-empty, countable and pairwise disjoint.
\end{prop}
\begin{proof}
Since $U_{k,0}\in X_k\subseteq \Q+\{U_{k,m} \colon 0\leq m<2^{k}\}$, the set $X_k$ is non-empty,
countable and, thanks to Proposition~\ref{prop:A3n}, we have $X_k\cap\overline\Q=\emptyset$. Now we show
that $X_k\cap X_h=\emptyset$ for all $h\neq k$. Assume that $\tau\in X_k\cap X_h$ with $h\leq k$. Thus,
by definition, there exists $m_1,m_2$ such that $s_n=f_k(n+m_1)=f_h(n+m_2)$ for all sufficiently large
$n$. In particular, $f_k(n)$ is periodic modulo $2^h$. Since $f_k(0)=1=-f_{k}(2^{k-1})$, one has
$h\geq k$ and so $h=k$.
\end{proof}

\subsection{Verifying whether $\tau\in X_k$ in a finite number of steps}
For each $k\geq 1$, the set $X_k$ is defined as the set of all $\tau\in\R$ whose greedy sequence is
eventually periodic with repeating block $B_k$. Since there is no effective bound on when the sequence
starts being periodic, one might then expect that there is no algorithm which verifies whether a given
number $\tau$ is in $B_k$. Surprisingly, however, such an algorithm can be constructed. We assume
that $\tau\not\in\Q$, since otherwise we already know $\tau\notin X_k$ for all $k$. Also, we assume that we
know $\tau$ to arbitrary precision and that we can tell whether two real numbers coincide (actually, this
is not needed if we already know that $\tau=r+U_{k,m}$ for some $k,m$ and some $r\in\Q$).

The algorithm is rather simple and proceeds as follows.
\begin{enumerate}
\item {\bf Determine whether $\tau\in X_1$. } In order to determine if $\tau\in X_1$, one starts by
      finding the first $N_1\geq 24$ such that $(s_{m})_{N_1\leq m< N_1+2^{1}}=B_{1}$. By
      Proposition~\ref{prop:A1n} we know that such $N_1$ exists and by the upper bound for the harmonic
      sum we also know that $N_1\leq \max(24,e^{|\tau|}+2)$. Then we claim that one has $\tau\in X_1$
      (and thus $\tau\notin X_k$ for all $k>1$) if and only if
\begin{equation*}
\tau = U_{1,m_1}+\sum_{1\leq n<N_1}\frac{s_n-f_1(n+m_1)}{n}
\end{equation*}
      where $0\leq m_1<2^{1}$, $m_1\equiv -N_1\pmod 2$. In particular, the algorithm stops if this
      equality is satisfied and otherwise it moves to the next step. To prove this equivalence, we first
      observe that by Corollary~\ref{cor:A4n} we have
\begin{equation*}
(s_{n})_{n\geq N_1}=(B_1,\delta_1\cdot B_1,\delta_2\cdot B_1,\dots)
\end{equation*}
      with $\delta_i\in \{\pm 1\}$. Also, we have $\tau\in X_1$ if and only if $\delta_i=1$ for all $i$.
      Indeed, if $\delta_i=1$ for all $i$ then by definition  $\tau\in X_1$. Conversely, if $\delta_i=-1$
      for some $i$ then $(s_{n})_{n\geq N_1}$ contains the block $(B_1,-B_1)$ and so, by
      Corollary~\ref{cor:A4n},
\begin{equation*}
(s_{n})_{n\geq N'}=(B_2,\delta_1'\cdot B_2,\delta_2'\cdot B_2,\dots).
\end{equation*}
      for some $N'\geq N_1$. In particular, the tail of $(s_n)_n$ cannot be periodic with repeating block
      $B_1$, i.e., $\tau\notin X_1$. Finally, one easily sees that  $\delta_i=1$ for all $i$, i.e.,
\begin{equation*}
(s_{n})_{n\geq N_1}=(B_1, B_1, B_1,\dots),
\end{equation*}
      if and only if
\begin{equation*}
\tau = \sum_{1\leq n<N_1}\frac{s_n}{n}+\sum_{n\geq N_1}\frac{f_1(n+m_1)}{n}
     = \sum_{1\leq n<N_1}\frac{s_n-f_1(n+m_1)}{n}+U_{1,m_1}.
\end{equation*}
\end{enumerate}
\smallskip

One then proceeds inductively. Assuming $\tau\notin X_{k-1}$ with $k\geq2$, to verify whether $\tau\in
X_k$ one proceeds as follows.

\smallskip
\begin{enumerate}
\item[(2)] {\bf Determine the first $N_k\geq 2^{k+2}(k+2)$ such that $(s_{m})_{N_k\leq m< N_k+2^{k}}=B_{k}$}.
      Since $\tau\notin X_{k-1}$ we know that such an $N_k$ exists. We can also provide an explicit bound
      for it in terms of $\tau$ and $(s_{n})_{n<N_{k-1}}$; we shall give the details at the end of the algorithm.

\item[(3)] {\bf Determine whether $\tau\in X_k$. } In the same way as in step (1), one has that $\tau\in X_k$ if
      and only if
\begin{equation*}
\tau = U_{k,m_k}+\sum_{1\leq n<N_k}\frac{s_n-f_k(n+m_k)}{n}
\end{equation*}
where $0\leq m_k<2^{k}$, $m_k\equiv -N_k\pmod {2^k}$.
\end{enumerate}
\smallskip

As anticipated, given $\tau\notin X_k$ and the first $N_{k-1}\geq 2^{k+1}(k+1)$ such that
$(s_{m})_{N_2\leq m< N_2+2^{k-1}}=B_{k-1}$, one can give an upper bound on $N_k$. Indeed, since
$\tau\notin X_{k-1}$ we know that
\begin{align*}
G_{\tau,N_{k-1}}:=\tau - U_{k-1,m_{k-1}}+ \sum_{n< N_{k-1}} \frac{f_{k-1}(n+m_{k-1})-s_n}{n}\neq0.
\end{align*}
Then we have $N_k\leq \max\big(N_{k-1}, 2^{k+2}(k+2), 2^k+4/|G_{\tau,N_{k-1}}|\big)$. To see this, we
observe that, by definition, for any $M\geq N_{k-1}$ we have
\begin{align*}
G_{\tau,N_{k-1}}
&= \sum_{N_{k-1}\leq n< M} \frac{s_n-f_{k-1}(n+m_1)}{n}
  + \sum_{n\geq M} \frac{s_n}{n}- \sum_{n\geq M} \frac{f_{k-1}(n+m_{k-1})}{n}.
\end{align*}
In particular,
\begin{align}\label{eq:effective}
\sum_{N_{k-1}\leq n< M} \frac{|s_n-f_{k-1}(n+m_{k-1})|}{n}
> |G_{\tau,N_{k-1}}|
 - \bigg|\sum_{n\geq M} \frac{s_n}{n}\bigg|
 - \bigg|\sum_{n\geq M} \frac{f_{k-1}(n+m_{k-1})}{n}\bigg|.
\end{align}
By Proposition~\ref{prop:A1n} and Corollary~\ref{cor:u} both sums on the right are bounded by $2/M$.
Thus,~\eqref{eq:effective} implies
\begin{equation*}
\sum_{N_{k-1}\leq n< M} \frac{|s_n-f_{k-1}(n+m_{k-1})|}{n}
> |G_{\tau,N_{k-1}}| - \frac{4}{M}.
\end{equation*}
Taking $M$ to be the largest integer such that $M\equiv N_{k-1}\pmod {2^{k-1}}$ and
$M<2^k+4/|G_{\tau,N_{k-1}}|$, we have that the right hand side is greater than zero. Thus, with this
choice we obtain
\begin{equation*}
\sum_{N_{k-1}\leq n< M} \frac{|s_n-f_{k-1}(n+m_{k-1})|}{n} > 0
\end{equation*}
and so $s_n\neq f_{k-1}(n+m_{k-1})$ for some $n \in [N_{k-1},M)$. Since $(s_n)_{n\geq N_{k-1}}$ can be
written in terms of blocks $\pm B_{k-1}$ it follows that $(s_n)_{N_{k-1}\leq n<M}$ has to contain a block
$(B_{k-1},-B_{k-1})=B_k$.
\begin{remark}
To give some examples, using the above algorithm we verified  that $U_{2,0}+r\in X_2$ if $r=1,2,3$ and
$U_{2,0}+r\notin X_2$ if $r=4,\dots, 10$.
\end{remark}

\section{Proof of Theorem~\ref{thm:asymp}}
\begin{defin}
For $\tau\notin Y_{\infty}$ and $h\geq 1$, let $n_h$ be the minimum integer such that  $|\tau -
\sigma_{n_{h}-1}(\tau)| < g_{h-1}(n_h)$ and $n_h\geq 2^{h} h$. Notice that by Corollary~\ref{cor:corblo}
and~\eqref{eq:signs_to_ineq_2} we know that such an integer exists.
\end{defin}
The following lemma gives some information on the sequence $(n_h)_h$ and in particular it put a rather
sharp limit for the possible sign patterns in $(s_n)$ for $n_h\leq n<n_{h+1}$.
\begin{lem}\label{lem:imp_thm2}
Assume that $\tau\notin (Y_{\infty}\cup\Q)$. The sequence $(n_{h})_h$ is non-decreasing and such that
$n_{h+1}\equiv n_h\pmod {2^h}$ for all $h\geq 1$; also, $n_1\leq 4 e^{|\tau|}$. Moreover, if $n_h<n_{h+1}$
we have
\begin{equation*}
(s_n)_{n_h\leq n< n_{h+1}}=(\eta_{0,h}\cdot B_{h},\eta_{1,h}\cdot B_h, \dots, \eta_{r_h,h} \cdot B_h),
\end{equation*}
where $r_h:=(n_{h+1}-n_h)/2^h-1$ and $\eta_{i,h}\in\{\pm 1\}$ is such that $\eta_{i,h}=s_{n_{h+1}}$ for
all $i \in [0,r_h]$ if $n_{h}\geq 2^{h+1}(h+1)$ and for all $i \in [h + 2,r_h]$ in any case.
\end{lem}
\begin{proof}
The fact that $n_h$ is non-decreasing follows immediately from  Corollary~\ref{cor:A2n}, whereas the
inequality $n_1\leq \max(e^{|\tau|}+1,4)\leq 4e^{|\tau|}$ follows from bounding the harmonic sum.

By Lemma~\ref{lem:A8n} and Corollary~\ref{cor:A4n} we have
\begin{equation*}
(s_n)_{n\geq n_h}=(\eta_{0,h}\cdot B_{h},\eta_{1,h}\cdot B_h, \dots),
\end{equation*}
for some $(\eta_{i,h})_{i\geq 0}$ with $\eta_{i,h}\in\{\pm 1\}$.  By Lemma~\ref{lem:A8n} we have
$(s_{n})_{n_{h+1}\leq n<n_{h+1}+2^{h+1}}=B_{h+1}$ and so it must be $n_{h+1}= n_h+ 2^{h}r_h$ for some
$r_h\geq 0$.

We claim that if $\eta_{j,h}\neq \eta_{j+1,h}$ for some $j \in [0,r_h]$ (with
$\eta_{r_{h}+1,h}:=\eta_{0,h+1}$), then $n_h+2^h j<2^{h+1}(h+1)$. Indeed, if $\eta_{j,h}\neq
\eta_{j+1,h}$, then Equation~\eqref{eq:signs_to_ineq_2} implies that $|\sigma_{n'-1}-\tau|< g_{h}(n')$
with $n':=n_h+2^h j$ (the inequality is strict since $\tau\notin\Q$). Since $n'<n_{h+1}$ it then follows
that it must be $n'=n_h+2^h j<2^{h+1}(h+1)$, as claimed. This implies in particular that if $n_{h}\geq
2^{h+1}(h+1)$ then $\eta_{i,h}=\eta_{i+1,h}$ for $0\leq i\leq r_h$ or, equivalently,
$\eta_{i,h}=s_{n_{h+1}}$ for $0\leq i\leq r_h$ since $\eta_{r_h+1,h}=s_{n_{h+1}}$. Similarly, if
$h+2\leq  i\leq r_h$ then $n_h+2^h i\geq 2^h h+ (h+2)2^h=2^{h+1}(h+1)$ and so one concludes as before.
\end{proof}
Thanks to the above lemma we have a good control on the sequence $(n_h)_h$.  We now need two combinatorial
lemmas. The first one is a well known upper bound for the number of lattice point in a
simplex~\cite{Beged-Dov}.
We give a simple proof for completeness.
\begin{lem}\label{lem:hyperpyramid}
Let $k,m\geq 1$ and $b_1,\dots,b_k\in\N$.
Then
\begin{equation*}
S:=|\{(a_1,\dots,a_k)\in\Z_{\geq 0}^k\colon b_1a_1+b_2a_2+\dots+b_ka_k = m\}|
\leq \frac{(m+b_2+\cdots+b_k)^{k-1}}{(b_2\cdots b_k)(k-1)!}.
\end{equation*}
\end{lem}
\begin{proof}
Clearly,
\begin{align*}
S&\leq |\{(a_2,\dots,a_k)\in\Z_{\geq 0}^{k-1}\colon b_2a_2+\dots+b_ka_k\leq m\}|\\
&\leq \int_{0}^\infty\cdots\int_{0}^\infty \chi_{[0,1]}\Big(\frac{b_2x_2+\dots+b_kx_k}{m+b_2+\cdots+b_k}\Big)\,\d x_2\cdots \,\d x_k.
\end{align*}
By a change of variables this is
\begin{equation*}
\frac{(m+b_2+\cdots+b_k)^{k-1}}{b_2\cdots b_k} \int_{0}^\infty\cdots\int_{0}^\infty \chi_{[0,1]}(x_2+\cdots+x_k)\,\d x_2\cdots \,\d x_k
=\frac{(m+b_2+\cdots+b_k)^{k-1}}{(b_2\cdots b_k)(k-1)!},
\end{equation*}
as desired.
\end{proof}
\begin{lem}\label{lem:ml2}
For $k\geq3$, $\ell\geq 0$ we have
\begin{equation*}
Z(k,\ell)
:=    \hspace{-0.5em}
      \sum_{a_1,\dots,a_{k-1}\geq 0\atop  a_{1}+\cdots+2^{k-2}a_{k-1}=\ell}
         \prod_{2\leq h\leq k}2^{h\chi_{[0,2^hh)}(a_1+2a_{2}+\cdots+2^{h-2}a_{h-1})}
\leq  2^{-\frac{k^2}{2}+5k}\frac{\ell^{k-2}}{(k-2)!}+2^{\frac{k^2}{2}+4k}k!.
\end{equation*}
\end{lem}
\begin{proof}
First, we observe that $Z(k,\ell)\leq 4Z'(k,\ell/2)$, where
\begin{equation*}
Z'(k,\ell)
 = \sum_{a_2,\dots,a_{k-1}\geq 0\atop  a_{2}+\cdots+2^{k-3}a_{k-1}\leq\ell}
     \prod_{3\leq h\leq k}2^{h\chi_{[0,2^{h-1}h)}(a_{2}+\cdots+2^{h-3}a_{h-1})}.
\end{equation*}
We shall prove that for all $k\geq3,\ell\geq 0$ one has
\begin{equation}\label{eq:ind}
Z'(k,\ell)\leq 2^{-\frac{k^2}{2}+5k-2}  \frac{\ell^{k-2}}{(k-2)!}+ 2^{\frac{k^2}{2}+ 4k -2} k!,
\end{equation}
from which the claimed inequality for $Z$ follows immediately. We prove this by induction over $k$. The
result is obvious if $k=3$. Now assume that the claimed inequality holds for $k-1\geq3$. We split $Z'$ into
$Z'=Z'_<+Z'_\geq$, where $Z'_<$ is the contribution of the terms with
$a_{2}+\cdots+2^{k-3}a_{k-1}<2^{k-1}k$. We have
\begin{align*}
Z'_<(k,\ell)
&\leq 2^k\sum_{a_2,\dots,a_{k-1}\geq 0,\ a_{k-1}<4k\atop  a_{2}+\cdots+2^{k-4}a_{k-2}< 2^{k-3}(4k-a_{k-1})}
         \prod_{3\leq h\leq k-1}2^{h\chi_{[0,2^{h-1}h)}(a_{2}+\cdots+2^{h-3}a_{h-1})}\\
&=    2^k\sum_{0\leq a<4k}Z'(k-1,2^{k-3}(4k-a)).
\end{align*}
Thus, by the inductive hypothesis,
\begin{align}
Z'_<(k,\ell)
&\leq\sum_{0\leq a<4k}2^{\frac{k^2+2k+3}{2}}\frac{(4k-a)^{k-3}}{(k-3)!}+2^{\frac{k^2+6k-11}{2}}2^k4k(k-1)!\notag\\
&\leq \frac{2^{\frac{k^2+2k+3}{2}}}{(k-3)!}\Big[\int_{0}^{4k}(4k-x)^{k-3}\, \d x + (4k)^{k-3}\Big]
      +2^{\frac{k^2+8k-7}{2}}k!                                                                           \notag\\
&=    2^{\frac{k^2+6k-9}{2}}\frac{k^{k-3}}{(k-3)!}\Big[\frac{4k}{k-2}+1\Big]+2^{\frac{k^2+8k-7}{2}}k!     \notag\\
&\leq 2^{\frac{k^2+9k-9}{2}}+2^{\frac{k^2+8k-7}{2}}k!,                                                    \label{eq:inda}
\end{align}
since $\frac{k^{k-3}}{(k-3)!}\big[\frac{4k}{k-2}+1\big]\leq 2^{3k/2}$ for all $k\geq 4$. In order to
bound $Z'_\geq$, we first observe we can assume that $\ell\geq 2^{k-1}k$, since otherwise $Z'_\geq$ is just
the empty sum. Now
\begin{align*}
Z'_\geq(k,\ell)
&\leq \sum_{0\leq a_{k-1}\leq \ell/ 2^{k-3}}\sum_{a_2,\dots,a_{k-2}\geq 0\atop  a_{2}+\cdots+2^{k-4}a_{k-2}\leq\ell-2^{k-3}a_{k-1}}
      \prod_{3\leq h\leq k-1}2^{h\chi_{[0,2^{h-1}h)}(a_{2}+\cdots+2^{h-3}a_{h-1})}\\
&\leq \sum_{0\leq a_{k-1}\leq \ell/ 2^{k-3}}Z'(k-1,\ell-2^{k-3}a_{k-1}) .
\end{align*}
Thus, by the inductive hypothesis we have
\begin{align}
Z'_\geq(k,\ell)
&\leq \sum_{0\leq a\leq \ell/ 2^{k-3}}\frac{(\ell-2^{k-3}a)^{k-3}}{(k-3)!}2^{-\frac{k^2-12k +15}{2}}
     +2^{\frac{k^2+6k-11}{2}}(\ell/ 2^{k-3}+1)(k-1)!                            \notag\\
&\leq \frac{2^{-\frac{k^2-12k +15}2}}{(k-3)!}
     \bigg(\int_0^{\ell/ 2^{k-3}}(\ell-2^{k-3}x)^{k-3}\,\d x + \ell^{k-3}\bigg)
    +\big(2^{\frac{k^2+4k-5}{2}}\ell +2^{\frac{k^2+6k-11}{2}}\big)(k-1)!        \notag\\
&=2^{-\frac{k^2-10k +15}2}  \frac{\ell^{k-2}}{(k-2)!}\Big[8+\frac{2^{k}(k-2)}\ell\Big]
     +\big(2^{\frac{k^2+4k-5}{2}}\ell +2^{\frac{k^2+6k-11}{2}}\big)(k-1)!       \notag\\
&\leq 10\cdot 2^{-\frac{k^2-10k +15}2}  \frac{\ell^{k-2}}{(k-2)!}
     +\big(2^{\frac{k^2+4k-5}{2}}\ell +2^{\frac{k^2+6k-11}{2}}\big)(k-1)!,      \label{eq:bfz}
\end{align}
since $\ell\geq 2^{k-1}k$. Now if $\ell\geq 2^{2k-1}k$ then
\begin{align*}
2^{\frac{k^2+4k-5}{2}}\ell (k-1)!
\leq 2^{\frac{k^2+4k-5}{2}}\ell^{k-2}\frac{2^{-2k(k-3)}k! }{k^{k-3}}
\leq 6 \ell^{k-2}2^{-\frac{k^2-10k+5}2}  2^{-k^2+3k}
\leq 2^{-\frac{k^2-10k+5}2} \frac{\ell^{k-2}}{(k-2)!}
\end{align*}
since $k!\leq 6k^{k-3}$ and $6(k-2)! \leq  2^{k^2-3k}$ for $k\geq4$. Thus, in any case
\begin{align*}
2^{\frac{k^2+4k-5}{2}}\ell (k-1)!
\leq  2^{-\frac{k^2-10k+5}2} \frac{\ell^{k-2}}{(k-2)!}+2^{\frac{k^2+8k-7}{2}}k!
\end{align*}
and so~\eqref{eq:ind} follows from~\eqref{eq:inda} and~\eqref{eq:bfz} since $2^{\frac{k^2+9k-9}{2}}
+2^{\frac{k^2+6k-11}{2}}(k-1)! \leq 2^{\frac{k^2+8k-8}{2}}k!$.
\end{proof}
The following two propositions give the crucial steps in proving Theorem~\ref{thm:asymp}. The first one
essentially shows that, for almost all $\tau$, the size of $n_k$ is about $2^k k$ for infinitely many $k$. The
second one shows that for almost all $\tau$ there are no $n\asymp 2^k$ with $|\tau - \sigma_{n-1}(\tau)|$
much smaller than $ g_k(n)$.
\begin{prop}\label{prop:upper}
Let $f:\R_{>0}\to\R_{>0}$ be a function such that $f(x)\to\infty$ as $x\to\infty$ and let
\begin{equation*}
\X_f :=\bigg\{\tau\in\R\colon
         \begin{array}{l}
           \textup{there are at most finitely many $k$ such that }\\[.4\baselineskip]
          |\tau - \sigma_{n-1}(\tau)| < g_k(n)\ \textup{for some } n\in[2^k k, f(k) 2^k k]
         \end{array}
       \bigg\}.
\end{equation*}
Then $\meas(\X_f)=0$.
\end{prop}
\begin{proof}
Let $\X_f':=\X_f\cap \R'$ where $\R':=\R\setminus  (Y_{\infty}\cup\Q)$. By Proposition~\ref{prop:X} we
have that $\meas(\X_f)=\meas(\X_f')$. Then we observe that
\begin{align*}
\X_f'
&\subseteq\bigcup_{m\in\N}\Big\{\tau\in\R'\colon
                                |\tau - \sigma_{n-1}(\tau)| > g_k(n)
                                \text{ for all }
                                k\in\N
                                \text{ and for all }
                                n \in [2^k k, f(k) 2^k k / m]
                          \Big\}\\
&=\bigcup_{m\in\N}\bigcup_{q\in\N}\bigcap_{k\in\N}\X_{m,q,k},
\end{align*}
where
\begin{equation*}
\X_{m,q,k}:=\bigg\{\tau\in\R'\cap[-q,q]\colon
                   \begin{array}{l}
                     |\tau - \sigma_{n-1}(\tau)| > g_h(n)\text{ for all } h\in [1,k]\\
                    \text{and for all } n \in [2^h h, {f(h) 2^h h}/{m}]
                   \end{array}
            \bigg\}.
\end{equation*}
In particular, in order to prove that $\meas(\X_f)=0$ it is sufficient to show that $\meas(\X_{m,q,k})\to
0$ as $k\to\infty$.

Let $\tau\in\X_{m,q,k}$ and let $n_1,n_2,\dots$ as in Lemma~\ref{lem:imp_thm2}; in particular, $n_1\leq 4
e^{q}$. Also, let $C$ be such that $f(x)\geq 4m$ for $x\geq C$ so that if $h\geq C$ then, by the
definition of $\X_{m,q,k}$, $n_{h}\geq 2^{h+1}(h+1)$.

We split $\X_{m,q,k}$ depending on the value of $n_k$:
\begin{equation*}
\X_{m,q,k} =\bigcup_{\ell \geq f(k)2^{k} k/m}V_{\ell,m,q, k},
\qquad
V_{\ell,m, q,k}:=\{\tau \in \X_{m,q,k} \colon n_k=\ell \}.
\end{equation*}
If $\tau\in V_{\ell,m,q, k}$, then, by definition, $|\tau-\sigma_{\ell-1}(\tau)|<g_{k-1}(\ell)$ and so
$\tau$ has distance less than $g_{k-1}(\ell)$ from one of the elements of the set
\begin{equation*}
R_{\ell,m,q, k}:=\{\sigma_{\ell-1}(\tau) \colon  \tau\in V_{\ell,m, q,k}\}.
\end{equation*}
This set has cardinality bounded by the number of possible choices of signs $(s_n)_{n<\ell}$. By
Lemma~\ref{lem:imp_thm2}, if $h\geq C$ then the signs in $(s_n)_{n_h\leq n< n_{h+1}}$ are completely
determined (depending on the sign of $s_{n_{h+1}}$ and thus eventually on $s_{n_k}=s_\ell$), whereas for
each $h<C$ the first $h+2\leq C+1$ signs in $(s_n)_{n_h\leq n< n_{h+1}}$ are free and the other ones are
determined. Also, there are at most $2^{n_1}\leq 2^{4 e^q}$ possibilities for $(s_n)_{n<n_1}$ and $2$ for
$s_{\ell}$. Finally, we recall that $n_1\leq n_2\leq \cdots\leq n_k=\ell$ and that $n_{h+1}\equiv
n_h\pmod{2^{h}}$ for all $h\geq 1$. We then obtain that $|R_{\ell,m, q,k}|$ is bounded by
$2^{C(C+1)+4e^q+1}$ times
\begin{align*}
&|\{(m_1,\dots, m_{k-1})\colon m_1\leq \dots \leq m_k:=\ell ,\ m_1\leq 4 e^q,\  m_{h+1}\equiv m_h\pmod {2^{h}}\ \text{for } 1\leq h<k\}|\\
&\qquad
 =     |\{(a_0,\dots, a_{k-1})\in\Z_{\geq 0}\colon a_0+ 2a_1+2^2a_{2}+\cdots+2^{k-1}a_{k-1}=\ell,\ a_0\leq 4e^q\}|\\
&\qquad
 \leq  4e^q \frac{(\ell+2^2+\cdots+2^{k-1})^{k-2}}{(2^2\cdots 2^{k-1})(k-2)!}
 \leq  4e^q  2 \frac{(\ell+2^k)^{k-2}}{2^{\binom{k}{2}}(k-2)!} \ll e^q  \frac{\ell^{k-2}}{2^{\binom{k}{2}}(k-2)!},
\end{align*}
by Lemma~\ref{lem:hyperpyramid}, where in the last inequality we used that $\ell=n_k\geq 2^kk$. It follows that
\begin{align*}
\meas(\X_{m,q, k})
&\leq     2\sum_{\ell \geq f(k)2^{k} k/m}|R_{\ell,m,q, k}|g_{k-1}(\ell)
 \ll_{q,m} \sum_{\ell \geq f(k)2^{k} k/m}\frac{\ell^{k-2}}{2^{\binom{k}{2}}(k-2)!}\cdot\frac{2^{\binom{k-1}{2}}(k-1)!}{\ell^k} \\
&=         \sum_{\ell \geq f(k)2^{k} k/m}\frac{k-1}{2^{k-1} \ell^2}\ll_{q,m}  \frac{1}{2^{2k}f(k)},
\end{align*}
by Lemma~\ref{lem:g_signs} (and for $k\geq C$). Thus, $\meas(\X_{m,q,k})\to 0$ as $k\to\infty$, as desired.
\end{proof}
\begin{prop}\label{prop:lower}
Let $f\colon\R_{\geq 0}\to\R_{>0}$ be a function such that $f(x)\gg 2^{5x}x^x$ for $x$ large enough and
let
\begin{equation*}
\Y_f:=\bigg\{\tau\in\R\colon
      \begin{array}{l}
       \textup{there exist infinitely many $k$ such that}\\[.4\baselineskip]
       |\tau - \sigma_{n-1}(\tau)| <{g_k(n)}/{f(k)}
       \textup{ for some } n \in [2^{k-1} (k-1), 2^k k)
      \end{array}
      \bigg\}.
\end{equation*}
Then $\meas(\Y_f)=0$.
\end{prop}
\begin{proof}
Let $\Y_f':=\Y_f\cap \R'$, with $\R':=\R\setminus (Y_{\infty}\cup\Q)$, so that
$\meas(\Y_f')=\meas(\Y_f)$. We have
\begin{equation*}
\Y_f' =
\bigcup_{q\in\N}\bigcap_{K\in\N}\bigcup_{k\geq K}\bigcup_{2^{k-1} (k-1) \leq \ell< 2^kk}\Y_{q,k,\ell},
\end{equation*}
where
\begin{equation*}
\Y_{q,k,\ell}:=\{\tau\in\R'\cap[-q,q]\colon |\tau - \sigma_{\ell-1}(\tau)| < g_k(\ell)/f(k)\}.
\end{equation*}
It suffices to show that
\begin{equation*}
\meas\bigg(\bigcup_{k\geq K}\bigcup_{2^{k-1} (k-1)  \leq \ell < 2^kk }\Y_{q,k,\ell}\bigg)\to0
\end{equation*}
as $K\to\infty$. Also, we can assume $f(x)\geq 1$ for $x\geq K$.

Let $\tau\in \Y_{q,k,\ell}$ and let $n_1,n_2,\dots$ be as in Lemma~\ref{lem:imp_thm2}. By
Corollary~\ref{cor:A2n} we have $n_{k-1}\leq \ell< n_{k}$. Notice also that Lemma~\ref{lem:imp_thm2} implies
that $\ell\equiv n_{k-1}\pmod{2^{k-1}}$.

The number $\tau$ belongs to $\Y_{q,k,\ell}$, so it has distance less than $\frac{g_{k}(\ell)}{f(k)}$ from one
of the elements of the set
\begin{equation*}
R_{q,k,\ell}:=\{\sigma_{\ell-1}(\tau) \colon  \tau\in \Y_{q,k,\ell}\}.
\end{equation*}
This set has cardinality bounded by the number of possible choices of signs $(s_n)_{n<\ell}$. By
Lemma~\ref{lem:imp_thm2}, for each $h\in[1,k-1]$ we have only one choice for the signs in $(s_n)_{n_h\leq
n< n_{h+1}}$ if $n_h\geq 2^{h+1}(h+1)$ and at most $2^{h+2}$ in any case. Also, we have at most $2^{4e^q}
$ for the signs in $(s_n)_{n< n_{1}}$. Thus, writing $\ell=n_{k-1}+2^{k-1}a_{k-1}$,
$n_h=n_{h-1}+2^{h-1}a_{h-1}$ for $2\leq h\leq k-1$ and $n_1=a_0$ for some $a_{0},\dots
a_{k}\in\Z_{\geq 0}$, we see that
\begin{align*}
|R_{q,k,\ell}|
&\ll_q \sum_{a_0,a_1,\dots,a_{k-1}\geq 0\atop a_0+ 2a_1+2^2a_{2}+\cdots+2^{k-1}a_{k-1}=\ell}
       \prod_{2\leq h<k}2^{(h+2)\chi_{[0,2^{h+1}(h+1))}(a_0+ 2a_1+2^2a_{2}+\cdots+2^{h-1}a_{h-1})}\\
&\ll_q 2^{k} \sum_{b_1,\dots,b_{k}\geq 0\atop b_1+ 2b_2+\cdots+2^{k-1}b_{k}=\ell}
       \prod_{3\leq h\leq k}2^{h\chi_{[0,2^{h}h)}(b_1+ 2b_2+\cdots+2^{h-2}b_{h-1})}.
\end{align*}
Applying Lemma~\ref{lem:ml2} we then obtain
\begin{align*}
|R_{q,k,\ell}|
&\ll_{q} 2^{-\frac{k^2}{2}+5k}\frac{\ell^{k-1}}{(k-1)!} + 2^{\frac{k^2}{2}+ 5k}(k+1)!
&\leq    2^{\frac{k^2}{2}+4k} \frac{k^k}{k!}+2^{\frac{k^2}{2}+6k}(k+1)!
\ll      2^{\frac{k^2}{2}+5k} k^{k+1}
\end{align*}
for $\ell\leq 2^kk$ and $k\geq3$ since $k!\ll k^k 2^{-k}$.
In particular, for $K\geq3 $ we have
\begin{align*}
\meas\bigg(\bigcup_{k\geq K}\bigcup_{2^{k-1} (k-1)  \leq \ell < 2^kk }\Y_{q,k,\ell}\bigg)
&\leq  \sum_{k\geq K}\sum_{2^{k-1} (k-1) \leq \ell< k2^k}|R_{q,k,\ell}|\frac{g_{k}(\ell)}{f(k)}\\
&\ll_q \sum_{k\geq K}\sum_{2^{k-1} (k-1) \leq \ell< k2^k} 2^{\frac{k^2}{2}+5k}k^{k+1} \frac{2^{\binom{k}2}k!}{f(k)\ell^{k+1}}\\
&\ll   \sum_{k\geq K} 2^{11k/2}\frac{k! \, k^{k}}{(k-1)^{k}f(k)}
\ll \sum_{k\geq K} \frac{2^{9k/2}k^k}{f(k)},
\end{align*}
by Lemma~\ref{lem:g_signs}. The sum goes to zero if $f(x)\gg 2^{5x}x^x$ and so the proof is complete.
\end{proof}
The following lemmas put a lower and an upper limit to how small $|\tau - \sigma_{m-1}(\tau)| $ can be
when $\tau$ is in $\R\setminus\Y_f$ and $\R\setminus\X_f$ respectively.
\begin{lem}\label{lem:lb}
Let $f\colon\R_{\geq 0}\to\R_{>0}$ be an increasing function with $ \log(f(x))\leq x\log x+O(x) $ as
$x\to\infty$. Then for all $\tau\notin\Y_f$ we have
\begin{equation*}
|\tau - \sigma_{m-1}(\tau)| \geq e^{-\frac{(\log m)^2}{\log 4}+O_\tau(\log m)}
\end{equation*}
as $m\to\infty$.
\end{lem}
\begin{proof}
If $\tau\notin\Y_f$ then there exists a sufficiently small $\delta>0$ such that $|\tau -
\sigma_{n-1}(\tau)| >\delta \frac{g_k(n)}{f(k)}$ for all $k\geq 1$ and all $n\in[2^{k-1} (k-1),2^k)$. Now
take any $m\geq2$ and let $k\in\N$ be such that $m\in[2^{k} k,2^{k+1}(k+1))$. Clearly such a $k$ exists
and satisfies $k+\log_2k\leq \log_2 m\leq  (k+1)+\log_2(k+1)$; in particular it follows that
\begin{equation*}
k=\log_2m-\log_2\log_2m+O(1)=\frac{\log m-\log\log m}{\log 2}+O(1)
\end{equation*}
as $m\to \infty$. The result then follows, since by Lemma~\ref{lem:g_signs} we have
\begin{equation*}
|\tau - \sigma_{m-1}(\tau)|
> \delta \frac{g_k(m)}{f(k)}>\frac{\delta\, k!\,2^{\binom{k}2}}{f(k)(2^{k+1}(k+1))^{k+1}}
= \frac{e^{-\frac{\log 2}2k^2+O_\tau(k)}}{f(k)}
> e^{-\frac{(\log m)^2}{\log 4} +O_\tau(\log m)}.
\end{equation*}
\end{proof}
\begin{lem}\label{lem:ub}
Let $f\colon \R_{>0}\to\R_{>0}$ be an increasing function with $f(x)\geq 3$ for $x$ large enough. Then for all
$\tau\notin\X_f$ there exist arbitrary large $n$ such that
\begin{equation*}
|\tau - \sigma_{m-1}(\tau)| \leq e^{-\frac{(\log m)^2-2\log m\log\log m}{\log 4}+O_\tau(\log m \log(f(\log n)))}
\end{equation*}
as $m\to\infty$.
\end{lem}
\begin{proof}
Let $\tau\notin\X_f$. Then there exists arbitrarily large $k,n\in\N$ such that $|\tau -
\sigma_{n-1}(\tau)| < g_k(n)$ and $2^k k\leq n\leq f(k) 2^k k$. In particular, $k\leq \log n$ and we have
\begin{equation*}
k=\frac{\log n-\log\log n}{\log 2}+O(\log(f(\log n))).
\end{equation*}
One then concludes as for Lemma~\ref{lem:lb}.
\end{proof}
\begin{proof}[Proof of Theorem~\ref{thm:asymp}]
Let $\Zz:=\R\setminus (\X_f\cup \Y_g)$ with $f(x)=\log x$ and $g(x)=2^{5x}x^x$. By
Propositions~\ref{prop:upper} and~\ref{prop:lower} we have $\meas(\X_f\cup \Y_g)=0$, whereas from
Lemma~\ref{lem:lb} and Lemma~\ref{lem:ub} we deduce that, for all $\tau\in \Zz$,
\begin{equation*}
0\leq \liminf_{m\to\infty}\frac{\log|\tau - \sigma_{m-1}(\tau)|+\frac{1}{\log 4}(\log m)^2}{\log m \log \log m}
\leq \frac{1}{\log 2},
\end{equation*}
which is Theorem~\ref{thm:asymp} in a stronger form.
\end{proof}
We conclude the section with the following proposition, which implies, in particular, that the conclusion
of Theorem~\ref{thm:asymp} does not hold for all $\tau\in\R$.
\begin{prop}\label{prop:limitation}
Given $f\colon \N \to \R_{>0}$ we can construct $\tau_f\in\R$ such that $|\tau_f - \sigma_{n}(\tau_f)|\leq
f(n)$ for infinitely many $n$. Moreover, the set of $\tau$ having this property is dense in $\R$.
\end{prop}
\begin{proof}
For simplicity we will only prove that there exists a $\tau_f$ with the required property. Along the same
lines one can show that the set of $\tau$ with this property is dense in $\R$.

We first observe that we can assume $f(n+1)<\tfrac1{4}f(n)$ for all $n\geq 1$ and that $f(1)<1$. Next we
observe that it suffices to construct a sequence of irrational numbers $(\tau_i)_i$ and an increasing
sequence of integers $(m_i)_i$ such that, for all $i\geq 1$,
\begin{equation}\label{eq:asdd}
 \sigma_n(\tau_{i+1})=\sigma_n(\tau_{i})\ \forall n\leq m_i,\qquad |\sigma_{m_i}(\tau_i)-\tau_i|<\tfrac12f(m_i),
 \qquad
 |\tau_{i+1}-\tau_{i}|<\tfrac14f(m_i) .
\end{equation}
Indeed, this implies that for all $j> i$
\begin{equation*}
|\tau_{j}-\tau_i|<\tfrac14(f(m_i)+\dots+f(m_{j-1}))< \tfrac1{2} f(m_i)
\end{equation*}
and thus
\begin{equation*}
|\tau_{j}-\sigma_{m_i}(\tau_i)|\leq |\tau_{i}-\sigma_{m_i}(\tau_i)| + |\tau_{j}-\tau_i|<f(m_i)
\end{equation*}
for all $i\leq j$. It follows that the limit $\tau_f:=\lim_{i\to\infty}\tau_i$ exists and satisfies
$|\tau_f-\sigma_{m_i}(\tau_f)|\leq f(m_i)$ for all $i$, as desired.

We thus just have to see that such sequences exist. To construct the two sequences one proceed as
follows. For $\tau_1$ one takes $m_1=1$ and any irrational $\tau_1$ in
$(1-\tfrac1{2}f(1),1+\tfrac1{2}f(1))$. Clearly, one has $|\tau_1-\sigma_1(\tau_1)|<f(1)$. Now assume that we
have two sequences $(\tau_r)_{r \leq i}$ and $(m_r)_{r\leq i}$ satisfying~\eqref{eq:asdd}; we need to
construct $\tau_{i+1}$ and $m_{i+1}$. For any $\tau\in\R$ and $m\in\N$ the set
\begin{equation*}
I_m(\tau):=\{\alpha\in\R\colon \sigma_{n}(\alpha)=\sigma_{n}(\tau),\ \forall{n\leq m}\}
\end{equation*}
is an interval of non-zero measure containing $\tau$.
Since $\sigma_{m}(\tau_i)\to \tau_i$  as $m\to\infty$ we can find $q>m_i$ such that
$\sigma_{{q}}(\tau_i)$ is in the interior of $I_{m_i}(\tau_i)$,
$|\sigma_{{q}}(\tau_i)-\tau_i|<\tfrac18f(m_i)$ and such that
$|\sigma_q(\tau_i)-\tau_i|<|\sigma_n(\tau_i)-\tau_i|$ for all $n<q$ (notice that
$\sigma_n(\tau_i)\neq\tau_i$ for all $n$ since $\tau_i\notin\Q$). Notice that this last inequality
implies that $\sigma_{{q}}(\tau_i)\in I_q(\tau_i)$. It follows that $I_{m_i}(\tau_i)\cap I_q(\tau_i)$ is
an interval of non-zero measure containing $\sigma_{{q}}(\tau_i)$, and so we can find an irrational
$\tau_{i+1}\in I_{m_i}(\tau_i)\cap I_q(\tau_i)$ with $|\tau_{i+1}-\sigma_{{q}}(\tau_i)|<\frac{1}{2}f(q)$.
Then, defining $m_{i+1}:=q$, we have that $\tau_{i+1}$ and $m_{i+1}$ have the required properties since
one has $\tau_{i+1},\tau_i\in I_{m_i}(\tau_i)$ and
\begin{align*}
& |\tau_{i+1}-\sigma_{{m_{i+1}}}(\tau_{i+1})|
  =   |\tau_{i+1}-\sigma_{{m_{i+1}}}(\tau_{i})|
  <   \tfrac{1}{2}f(m_{i+1})\\
& |\tau_{i+1}-\tau_i|\leq |\sigma_{{q}}(\tau_i)-\tau_i|+|\sigma_{{q}}(\tau_i)-\tau_{i+1}|
  <   \tfrac{1}{8}f(m_i)+\tfrac12f(m_{i+1})
 \leq \tfrac{1}{4}f(m_i),
\end{align*}
where the equality on the first line follows since $\tau_{i+1}\in I_{m_{i+1}}(\tau_i)$.
\end{proof}

\section{The Thue--Morse constant}\label{sec:A8}
Let the family of sequences $\Eps_k:=(\Eps_k(n))_{n \geq 0}$ be defined recursively by
\begin{equation*}
\Eps_{0}(n) := \eps_n
\quad \text{ and }\quad
\Eps_{k+1}(n) := \sum_{m \leq n} \Eps_k(m)
\quad \text{ for } k, n \geq 0,
\end{equation*}
so that the $k$-th sequence is the sequence of partial sums of the $(k-1)$-th sequence, the first one
being the Thue--Morse sequence. The next lemma gives a description of the sequences $\Eps_k$.
\begin{prop}\label{prop:A4n}
For every $k \geq 0$ there exists a finite sequence $W_k:=(w_k(n))_{n=0}^{2^k - 1}$, with
$w_k(n)\in\Z_{\geq 0}$, such that:
\begin{equation}\label{eq:exp_Eps}
\Eps_k =(\eps_0\cdot W_k,\eps_1 \cdot W_k,\eps_2\cdot W_k,\dots).
\end{equation}
Moreover, one has:
\begin{enumerate}[$(i)$]
\setlength\itemsep{0.5em}
\item $\sum_{n=0}^{2^k-1} w_k(n) = 2^{\binom{k}{2}}$;
\item $w_k(2^k - 1) = \cdots = w_k(2^k - k) = 0$;
\item $w_k(2^k-k-1-n)=w_k(n)$ for all $n$, with $0\leq n\leq 2^{k}-k-1$.
\item $w_k(n) \leq 2^{\binom{k-1}{2}}$ for all $n \in [0,2^k)$ and the equality holds if and
      only if $2^{k-1}-k-1\leq n\leq 2^{k-1}-1$.
\end{enumerate}
\end{prop}
\begin{proof}
We give a proof using generating functions. Given a sequence of integers $(a_n)_{n \geq 0}$ with
generating function $F(x)$, the generating function of $\big(\sum_{m \leq n} a_n\big)_{n \geq 0}$ is
equal to $F(x)/(1 - x)$. Now, letting $E_k(x)$ denote the generating function of $\Eps_k$ for $k\geq 0$,
we have
\begin{equation*}
E_0(x) := \sum_{n = 0}^\infty \eps_n x^n = \prod_{j = 0}^\infty (1 - x^{2^j})
\end{equation*}
and thus
\begin{equation}\label{eq:A34n}
E_k(x)
= \frac{E_0(x)}{(1 - x)^k}
= \prod_{j = 0}^{k-1} \frac{1 - x^{2^j}}{1 - x} \prod_{j = k}^\infty (1 - x^{2^j})
= Q_k(x) E_0(x^{2^k}),
\end{equation}
where
\begin{equation}\label{eq:gs}
Q_k(x)
:= \prod_{j = 0}^{k - 1} (1 + x + x^2 + \cdots + x^{2^j - 1})
= \sum_{n = 0}^{2^k - 1} w_k(n) x^n,
\end{equation}
for some sequence of integers $(w_k(n))_{n = 0}^{2^k - 1}$ with
\begin{equation}\label{eq:def_w}
w_k(n):=\#\{(a_0,\dots,a_{k-1})\in\Z_{\geq 0}^{k}\colon a_0+\dots+a_{k-1}=n,\ a_j\leq 2^{j}-1\ \forall j\}.
\end{equation}
At this point,~\eqref{eq:exp_Eps} follows immediately from~\eqref{eq:A34n}. Also, $Q_k(1) =
2^{\binom{k}{2}}$ implies~$(i)$, and $\deg Q_k = 2^k - k - 1$ implies~$(ii)$. Furthermore, we have
$Q_k(x)=Q_k(1/x)x^{2^{k}-k-1}$ and so $(iii)$ follows. Finally, if $n$ satisfies
\begin{equation*}
\sum_{j=0}^{k-2}(2^j-1)=2^{k-1}-k-1\leq n\leq 2^{k-1}-1,
\end{equation*}
then for any $a_0,\dots,a_{k-2}$ as in~\eqref{eq:def_w} there exists a unique $a_{k-1}$ such that
$a_0+\dots+a_{k-1}=n$. Thus, for all such $n$ one has $w_k(n)=\prod_{j=0}^{k-2}{2^j}=2^{\binom{k-1}{2}}$
and the same reasoning gives $w_k(n)<2^{\binom{k-1}{2}}$ for all the other possible values of $n$, so
that $(iv)$ follows.
\end{proof}
\begin{cor}\label{cor:A5n}
For $n \geq 1$ let $n=2^\mu n'$ with $n'$ odd. Then $\Eps_1(n-1) = \Eps_2(n-1) = \cdots = \Eps_{\mu}(n-1)
= 0$ and $\Eps_{\mu+1}(n-1) =- 2^{\binom{\mu}{2}}\eps_{n}$.
\end{cor}
\begin{proof}
By~\eqref{eq:exp_Eps} we have $|\Eps_k(n-1)|=|\Eps_k(m-1)|$ if $m\equiv n\pmod{2^{k}}$. In particular,
if $\mu\leq k$, then $|\Eps_k(n-1)|=|\Eps_k(2^k-1)|=0$ by Proposition~\ref{prop:A4n}, $(ii)$. Moreover
writing $n'=2n''+1$ we have $n-1=2^{\mu+1}n''+2^\mu-1$. Thus, by~\eqref{eq:exp_Eps} and
Proposition~\ref{prop:A4n}, $(iv)$, we have $\Eps_{\mu+1}(n-1) = 2^{\binom{\mu}{2}}\eps_{n''} =
-2^{\binom{\mu}{2}}\eps_{n'}=-2^{\binom{\mu}{2}}\eps_{n}$ (because the multiplication by $2$ does not
modify the number of non-zero digits in the binary representation).
\end{proof}
The first few sequences $W_k$ are as follows:
\begin{align*}
&W_0=(1), & \hspace{-8em}W_1=(1,0),\hspace{8.5em}W_2=(1,1,0,0), \\
&W_3=(1,2,2,2,1,0,0,0), & \hspace{-2em}W_4=(1,3,5,7,8,8,8,8,7,5,3,1,0,0,0,0),\\
&W_5=&\hspace{-12.3em}(1,4,9,16,24,32,40,48,55,60,63,64,64,64,64,64,63,\dots,1,0,\dots,0).
\end{align*}
\begin{remark}
The sequence arising from the $W_k$ also appeared and was studied in the recent
work~\cite[Section~5]{VignatWakhare}. In this paper Vignat and Wakhare defined the numbers $\alpha_m^N$
such that
\begin{equation*}
\sum_{m_1=0}^{2^1-1}\cdots \sum_{m_N=0}^{2^N-1}g(m_1+\cdots+m_N)=\sum_{m=0}^{2^N-N-1}\alpha_m^Ng(m).
\end{equation*}
for any function $g$. It is not difficult to see (cf.~\eqref{eq:gs} and~\cite[(5.1)]{VignatWakhare}) that
the two sequences coincide, i.e., $w_{k-1}(n)=\alpha^k_n$ for $k\geq 1$, $0\leq n< 2^k$.
\end{remark}
\begin{remark}
The vectors $W_k$ are related to the Fabius function.
This can be defined as the unique solution $F : [0,1] \to \mathbb{R}$ of the following functional differential equation problem
\begin{align*}
F(0) &= 0; \\
F(1 - x) &= 1 - F(x), \quad\; x \in [0,1]; \\
F^\prime(x) &= 2 F(2x), \quad x \in [0,1/2].
\end{align*}
It can also be defined as the cumulative distribution function of $\sum_{n=1}^\infty X_n / 2^n$, where $X_1, X_2, \dots$ are independent and uniformly distributed random variables on the unit interval $[0,1]$.
It is an example of an infinitely differentiable function that is nowhere analytic~\cite{MR2329220,MR0197656}. 
Setting $F_k^\prime(n / 2^k) := 2^{(3k - k^2) / 2} w_k(n)$, for all integers $k \geq 0$ and $n \in [0,2^k)$, we have the identities
\begin{align*}
&\frac1{2^k} \sum_{n = 0}^{2^k - 1} F_k^\prime (n / 2^k) = 1; \\
&F_k^\prime\!\left(1 - \frac{n + k + 1}{2^k}\right) = F_k^\prime\left(\frac{n}{2^k}\right), \qquad\qquad\qquad\quad\;\; n\in[0,...,2^k - k - 1]; \\
&F_k^\prime\!\left(\frac{n}{2^k}\right) = 2 \cdot \frac1{2^{k-1}} \sum_{m \leq n} F_{k-1}^\prime\!\left(\frac{m}{2^{k-1}}\right), \quad n \in [0,..., 2^{k-1} - 1];
\end{align*}
which are discrete versions of
\begin{align*}
&\int_0^1 F^\prime(x)\d x = 1;\\
&F^\prime(1 - x) = F^\prime(x), \quad\, x \in [0,1]; \\
&F^\prime(x) = 2 \int_0^{2x} F^\prime(t) \d t, \quad x \in [0,1/2].
\end{align*}
Indeed, $F_k^\prime(x)$ approximates $F^\prime(x)=2F(2x)$ as $k \to +\infty$ (see Figure~\ref{fig:fabius}).
See~\cite{ReynaPreprint2,ReynaPreprint1,HauglandPreprint} for more information about the values of the Fabius function at dyadic fractions.

\begin{figure}[h]
\includegraphics[width=0.485\textwidth]{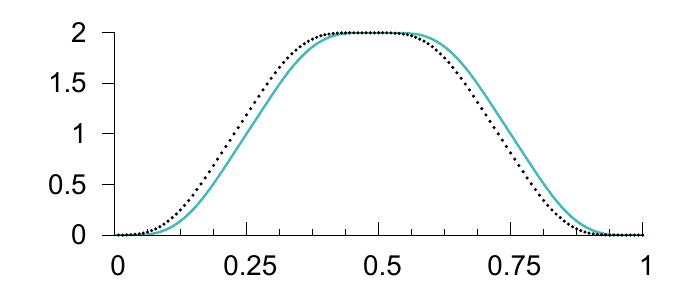}
\hspace{0.2cm}
\includegraphics[width=0.485\textwidth]{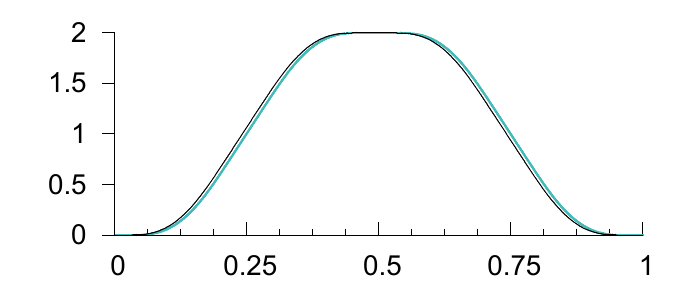}
\caption{The graphs of the derivative of the Fabius function (solid lines) and of the vectors
$W_6$ (left, dotted) and $W_9$ (right, dotted). The vectors are scaled to fit in $[0,2]$.}
\label{fig:fabius}
\end{figure}

\end{remark}

Proposition~\ref{prop:A4n} shows that the sequence of partial sums $\Eps_1$ is bounded. In particular, by
partial summation the series
\begin{equation}\label{eq:alt_tau}
\tau_0:=\sum_{n=1}^{+\infty} \frac{\eps_{n-1}}{n}
       =\sum_{n=1}^{+\infty} \Eps_1(n-1)\bigg(\frac{1}{n}-\frac{1}{n+1}\bigg)
\end{equation}
converges. Applying partial summation repeatedly, one also obtains the identity
\begin{equation}\label{eq:A35n}
\tau_0 - \sum_{m=1}^n\frac{\eps_{m-1}}{m}
= \sum_{m>n}\frac{\eps_{m-1}}{m}
= -\sum_{\ell=1}^k \Eps_\ell(n-1)G_{\ell-1}(n+1)
  + \sum_{m>n} \Eps_k(m-1)G_k(m),
\end{equation}
for all $k\geq 0$ and $n\geq 1$, where $G_k$ is defined by the recurrence relation
\begin{equation*}
G_0(x):= \frac{1}{x},
\quad
\text{and}
\quad
G_{k+1}(x) := G_k(x) - G_k(x+1)
\qquad
\text{for $k\geq 1$}.
\end{equation*}
Notice that the definition is similar to the one of $g_k$, but the shift in the second argument is
different. The following lemma shows that $G_k$ can be written also as a simple rational function.
\begin{lem}\label{lem:A6}
For every $k$ we have
\begin{align*}
G_k(x) = \frac{k!}{\prod_{\ell=0}^k (x+\ell)}.
\end{align*}
In particular, $0 < G_k(x)\leq k!/x^{k+1}$ when $x>0$ and $G_k(x)x^{k+1} = k!(1 + O(1/x))$ as $x\to\infty$.
\end{lem}
\begin{proof}
We proceed by induction on $k$. The formula is trivial for $k=0$. Using the recursive definition of $G_k$
and the inductive hypothesis we get:
\begin{align*}
G_{k+1}(x)
&= G_k(x) - G_k(x+1)
 = \frac{k!}{\prod_{\ell=0}^k (x+\ell)} - \frac{k!}{\prod_{\ell=0}^k (x+\ell+1)}\\
&= \frac{k!}{\prod_{\ell=0}^k (x+\ell)} - \frac{k!}{\prod_{\ell=1}^{k+1} (x+\ell)}
 = \frac{(k+1)!}{\prod_{\ell=0}^{k+1} (x+\ell)}.\qedhere
\end{align*}
\end{proof}
We now give the following lemma which implies that the series defining $\tau_0$ gives its greedy
representation.
\begin{lem}\label{lem:A7}
We have $\tau_0>0$. Moreover, for $n \geq 1$ let $n=2^\mu n'$ with $n'$ odd. Then
\begin{equation}\label{eq:sign_tau}
\eps_{n}\sum_{m>n}^\infty\frac{\eps_{m-1}}{m}> 2^{\binom{\mu}{2}} G_{\mu+1}(n+2^{\mu+1})>0.
\end{equation}
\end{lem}
\begin{proof}
Since $|\Eps_1(n)|\leq 1$ for all $n$ and $\Eps_1(0)=1$, $\Eps_1(1)=0$, by~\eqref{eq:alt_tau} we have
\begin{equation*}
\tau_0 \geq \frac{1}{2} - \sum_{n=3}^{+\infty} \bigg(\frac{1}n-\frac{1}{n+1}\bigg)
=\frac{1}{6}.
\end{equation*}
In particular $\tau_0>0$.
\smallskip

Applying~\eqref{eq:A35n} with $k=\mu+1$ gives
\begin{equation}\label{eq:A35bn}
\sum_{m>{n}}\frac{\eps_{m-1}}{m}
= -\, \Eps_{\mu+1}(n-1)G_{\mu}(n+1)
  + \sum_{m>n} \Eps_{\mu+1}(m-1)G_{\mu+1}(m),
\end{equation}
since by Corollary~\ref{cor:A5n} one has $\Eps_{\ell}(n-1)=0$ for all $\ell\leq\mu$. By
Proposition~\ref{prop:A4n}, $(iv)$, and the periodicity of $(|\Eps_{\mu+1}(n)|)_n$, we have
$|\Eps_{\mu+1}(m)|\leq 2^{\binom{\mu}{2}}$ and $\Eps_{\mu+1}(c-1)=0$ where $c$ is the smallest multiple
of $2^{\mu+1}$ such that $c>n$; in particular, $c<n+2^{\mu+1}$. Thus,
\begin{equation*}
\Big|\sum_{m>n} \Eps_{\mu+1}(m-1)G_{\mu+1}(m)\Big|
\leq 2^{\binom{\mu}{2}}\Big(\sum_{m>n} G_{\mu+1}(m)-  G_{\mu+1}(c)\Big)
= 2^{\binom{\mu}{2}} \big(G_{\mu}(n+1)- G_{\mu+1}(c)\big),
\end{equation*}
where in the last two steps we have used the positivity and the telescoping property of $G_{\mu+1}$.
Also, by Corollary~\ref{cor:A5n} we have $\Eps_{\mu+1}(n-1)=-\eps_{n}\, 2^{\binom{\mu}{2}} $.
Thus, \eqref{eq:A35bn} gives
\begin{equation*}
2^{\binom{\mu}{2}} G_{\mu}(n+1)-\eps_{n} \sum_{m>{n}}\frac{\eps_{m-1}}{m}
< \Big|\sum_{m>n} \Eps_{\mu+1}(m-1)G_k(m)\Big|
< 2^{\binom{\mu}{2}} \big(G_{\mu}(n+1)- G_{\mu+1}(c)\big)
\end{equation*}
and the claimed inequality follows since $G_{\mu+1}(c)\geq G_{\mu+1}(n+2^{\mu+1})$.
\end{proof}
We are now in position to prove the claims in Theorem~\ref{thm:A5}.
\begin{proof}[Proof of Theorem~\ref{thm:A5}]
The fact that $s_n=s_n(\tau_0)=\eps_{n-1}$ for all $n\geq 1$ follows from Lemma~\ref{lem:A7}, by proceeding
as in the proof of Corollary~\ref{cor:u}.
\smallskip

Let $k\geq 1$ and let $n=2^k n'$ with $n'$ odd. Corollary~\ref{cor:A5n} shows that
$\Eps_1(n-1)=\cdots=\Eps_k(n-1)=0$. Also, Proposition~\ref{prop:A4n} gives that $|\Eps_k(m)|\leq
2^{\binom{k-1}{2}}$ for all $m$. Hence,~\eqref{eq:A35n} gives
\begin{equation*}
|\tau_0 - \sigma_n|
=    \Big|\sum_{m>n} \Eps_k(m-1)G_k(m)\Big|
\leq 2^{\binom{k-1}{2}}\sum_{m>n} G_k(m),
\end{equation*}
since $G_k(m)$ is positive. Recalling the telescoping definition of $G_k$ and the bound for $G_{k-1}$
given in Lemma~\ref{lem:A6}, we get
\begin{equation}\label{eq:uto}
|\tau_0 - \sigma_n|
\leq 2^{\binom{k-1}{2}}G_{k-1}(n+1)
\leq 2^{\binom{k-1}{2}}(k-1)!/n^k,
\end{equation}
which is the first claim, since $2^{\binom{k-1}{2}}(k-1)! = c_{k-1}$. Moreover, applying~\eqref{eq:A35n}
with $k+2$ in place of $k$, gives
\begin{equation*}
\tau_0 - \sigma_n
= -\, \Eps_{k+1}(n-1)G_k(n+1)
  - \Eps_{k+2}(n-1)G_{k+1}(n+1)
  + \sum_{m>n} \Eps_{k+2}(m-1)G_{k+2}(m).
\end{equation*}
Corollary~\ref{cor:A5n} shows that $\Eps_{k+1}(n-1)=-2^{\binom{k}{2}}\eps_{n}$, whereas, telescoping,
again, we have
\begin{equation*}
\Big|\sum_{m>n} \Eps_{k+2}(m-1)G_{k+2}(m)\Big|< 2^{\binom{k+1}{2}} G_{k+1}(n+1)=O_k(n^{-k-2})
\end{equation*}
Thus,
\begin{equation*}
\tau_0 - \sigma_n
= \frac{\eps_{n} c_{k}}{n^{k+1}}(1+O_k(1/n))\sim_k\frac{\eps_{n} c_{k}}{n^{k+1}} .
\end{equation*}
Finally, by~\eqref{eq:sign_tau} we have
\begin{equation*}
|\tau_0 - \sigma_n|
> 2^{\binom{k}{2}} G_{k+1}(n+2^{k+1})
> 2^{\binom{k}{2}} G_{k+1}(3n)
> 2^{\binom{k}{2}-2k-2}\frac{(k+1)!}{n^{k+2}}
\end{equation*}
since $n\geq 2^k\geq (k+1)$. Thus, after  a quick computation with Stirling's formula one obtains
\begin{equation*}
\log|\tau_0 - \sigma_n|
> \min_{k\in\R}\Big(\frac{\log 2}{2}(k^2-5k)+\log(k!)-2-k\log n\Big)
= -\frac{(\log n-\log\log n)^2}{\log 4}+O(\log n).
\end{equation*}
Equation~\eqref{eq:asymptau0} then follows since~\eqref{eq:uto} with $n=2^k$ gives
\begin{equation*}
|\tau_0 - \sigma_n|
\leq 2^{-\frac{k^2+3k}{2}+1}(k-1)!
=    e^{-\frac{(\log n-\log\log n)^2}{\log 4} + O(\log n)}.
\qedhere
\end{equation*}
\end{proof}

\bibliographystyle{amsplain}
\providecommand{\MR}{\relax\ifhmode\unskip\space\fi MR }

\end{document}